\documentclass[11pt]{amsart}

\usepackage{graphicx}
\usepackage{amssymb}
\usepackage{amsmath}
\usepackage{enumerate}
\usepackage{amsfonts}
\usepackage{amsthm}
\usepackage[all,import]{xy}
\usepackage{xcolor}
\usepackage{url}
\makeatletter
\def\url@leostyle{%
  \@ifundefined{selectfont}{\def\UrlFont{\sf}}{\def\UrlFont{\small\ttfamily}}}
\makeatother
\numberwithin{equation}{section}
\usepackage[colorlinks=true, pdfstartview=FitV, linkcolor=black, 
            citecolor=black, urlcolor=black]{hyperref}
\usepackage{bookmark}
\usepackage[titletoc]{appendix} %For some reason, putting the bookmark and appendix packages after hyperref gets rid of the \endcsname error.
\theoremstyle{definition}
\newtheorem{prop}{Proposition}[section]

\newtheorem{theorem}[prop]{Theorem}

\newtheorem{corollary}[prop]{Corollary}
\newtheorem{conjecture}[prop]{Conjecture}

\newtheorem{defn}[prop]{Definition}
\newtheorem{example}[prop]{Example}
\newtheorem{notation}[prop]{Notation}
\newtheorem{convention}[prop]{Convention}

\newtheorem{assumption}[prop]{Assumption}
\newtheorem{caution}[prop]{Caution}
\newtheorem{remark}[prop]{Remark}
\newtheorem{image}[prop]{Figure} % Added for indexing figures more easily.

\newcommand{\nc}{\newcommand}
\nc{\DMO}{\DeclareMathOperator}	

\nc{\newnotation}{\nomenclature}
\nc{\wrap}{\cW}
\nc{\Cob}{\mathsf{Cob}}
\nc{\mul}{\mathsf{Mul}}
\nc{\fat}{\mathsf{fat}}
\nc{\cob}{\mathsf{Cob}}
\nc{\coh}{\mathsf{Coh}}
\nc{\idem}{\mathsf{Idem}}
\nc{\sets}{\mathsf{Sets}}
\nc{\near}{\mathsf{near}}
\nc{\sing}{\mathsf{Sing}}
\nc{\symp}{\mathsf{Symp}}
\nc{\perf}{\mathsf{Perf}}
\nc{\ssets}{\mathsf{sSets}}
\nc{\cmpct}{\mathsf{cmpct}}
\nc{\finite}{\mathsf{Finite}}
\nc{\compact}{\mathsf{cmpct}}
\nc{\pwrap}{\mathsf{PWrap}}
\nc{\coder}{\mathsf{Coder}}
\nc{\bimod}{\mathsf{Bimod}}
\nc{\grmod}{\mathsf{GrMod}}
\nc{\spaces}{\mathsf{Spaces}}
\nc{\pwrms}{\mathsf{PWrFuk}_{M,S}}
\nc{\pwrmf}{\mathsf{PWrFuk}_{M,F}}
\nc{\pwrapmf}{\mathsf{PWrFuk}_{M,F}}
\nc{\fuk}{\mathsf{Fukaya}}
\nc{\infwr}{\mathsf{InfWr}}
\nc{\fukaya}{\mathsf{Fukaya}}
\nc{\autml}{\mathsf{Aut}_{M,\Lambda}}
\nc{\fukml}{\mathsf{Fukaya}_{M,\Lambda}}
\nc{\fukmle}{\mathsf{Fukaya}_{M,\Lambda,\epsilon}}
\nc{\fukmod}{\wrfukcompact(M)\modules}
\nc{\lag}{\mathsf{Lag}}
\nc{\lagm}{\lag_M}
\nc{\lago}{\lag^o}
\nc{\lagml}{\lag_{M,\Lambda}} % For when I get lazy.
\nc{\lagmle}{\lag_{M,\Lambda,\epsilon}}
\nc{\fun}{\mathsf{Fun}}
\nc{\vect}{\mathsf{Vect}}
\nc{\chain}{\mathsf{Chain}}
\nc{\wrfuk}{\mathsf{WrFukaya}}
\nc{\wrfukcompact}{\mathsf{WrFukaya}_{\mathsf{cmpct}}}
\nc{\pwrfuk}{\mathsf{PWrFukaya}}
\nc{\inffuk}{\mathsf{InfFuk}}
\nc{\pwrfukml}{\mathsf{PWrFukaya}_{M,\Lambda}}
\nc{\inffukml}{\mathsf{InfFuk}_{M,\Lambda}}
\nc{\nattrans}{\mathsf{NatTrans}}
\nc{\corres}{\mathsf{Corres}}
\nc{\fukep}{\fukaya_\Lambda(M,\epsilon)}
\nc{\fukepop}{\fukaya_\Lambda(M,\epsilon)^{\op}}
\nc{\lagep}{\lag_\Lambda(M,\epsilon)}
\DMO{\cyl}{cyl} % Cylindrical
\nc{\dbcoh}{D^b\mathsf{Coh}}
\nc{\corr}{\mathsf{Corr}}

\nc{\cat}{\mathsf{Cat}}
\nc{\Cat}{\mathsf{Cat}}
\nc{\ainfty}{\mathsf{A}_\infty}
\nc{\inftycat}{\mathcal{C}\!\operatorname{at}_\infty}
\nc{\Ainftycat}{\mathcal{C}\!\operatorname{at}_{A_\infty}}
\nc{\ainftycat}{\mathcal{C}\!\operatorname{at}_{A_\infty}}
\nc{\stablecat}{\mathcal{C}\!\operatorname{at}_\infty^{\Ex}}

\DMO{\im}{im}
\DMO{\ev}{ev}
\DMO{\inj}{inj}
\DMO{\fib}{fib}
\DMO{\conf}{Conf}
\DMO{\chains}{Chains}
\DMO{\cochains}{Cochains}
\DMO{\cone}{Cone}
\DMO{\ran}{Ran}
\DMO{\rot}{Rot}
\DMO{\leg}{Leg}
\DMO{\imm}{imm}
\DMO{\adj}{adj}
\DMO{\tree}{Tree}
\DMO{\cube}{Cube}
\DMO{\deep}{deep}
\DMO{\back}{back}
\DMO{\front}{front}
\DMO{\flow}{Flow}
\DMO{\floer}{Floer}
\DMO{\maps}{Maps}
\DMO{\exact}{exact}
\DMO{\excess}{Excess}
\DMO{\Decomp}{Decomp}
\DMO{\decomp}{Decomp}
\DMO{\collar}{collar}
\DMO{\yoneda}{Yoneda}
\DMO{\hamspace}{Ham}
\DMO{\sympspace}{Symp}
\DMO{\holomaps}{Holomaps}
\DMO{\comp}{Comp}
\DMO{\crit}{Crit}
\DMO{\test}{{test}}
\DMO{\sign}{sign}
\DMO{\topp}{top}
\DMO{\indx}{Index}
\DMO{\Break}{Break} % Partitions
\DMO{\zero}{zero} %Zero
\DMO{\ob}{Ob}
\DMO{\gr}{Gr} % Grassmanian
\DMO{\Gr}{Gr} % Grassmanian
\DMO{\cl}{Cl} % Clifford Algebra
\DMO{\grlag}{GrLag}
\DMO{\GrLag}{GrLag}
\DMO{\Pin}{Pin}
\DMO{\Graph}{Graph}
\DMO{\grph}{Graph}
\DMO{\pin}{Pin}
\DMO{\gap}{Gap}
\DMO{\Ex}{Ex}
\DMO{\id}{id}
\DMO{\End}{End}
\DMO{\sym}{Sym} 
\DMO{\aut}{Aut}
\DMO{\DK}{DK} %Dold-Kan
\DMO{\poly}{poly} % Polynomial deRham forms
\DMO{\diff}{Diff}
\DMO{\coll}{coll}
\DMO{\dist}{dist} %Distance function
\DMO{\coker}{coker} %Cokernel
\nc{\kernel}{\ker} %Kernel
\DMO{\sspan}{span}
\DMO{\hocolim}{hocolim}	
\DMO{\holim}{holim}
\DMO{\sk}{sk}

\DMO{\ho}{ho}
\DMO{\fin}{fin}
\DMO{\tor}{Tor}
\DMO{\ext}{Ext}
\DMO{\ret}{Ret}
\DMO{\ham}{Ham}
\DMO{\con}{con}
\DMO{\leaf}{leaf}
\DMO{\supp}{supp}
\DMO{\edge}{edge}
\DMO{\colim}{colim}
\DMO{\edges}{edges}
\DMO{\Image}{image}
\DMO{\roots}{roots}
\DMO{\height}{height}
\DMO{\finmod}{FinMod}
\DMO{\leaves}{leaves}
\DMO{\planar}{planar}
\DMO{\vertices}{vertices}

\nc{\lagg}{\lag^{\cG}}
\nc{\iso}{\mathsf{Iso}}
\nc{\Set}{\mathsf{Set}}
\nc{\ass}{\mathsf{ \bf Ass}}
\nc{\Mod}{\mathsf{Mod}}
\nc{\modules}{\mathsf{Mod}}
\nc{\sset}{\mathsf{sSet}}
\nc{\liou}{\mathsf{Liou}}
\nc{\poset}{\mathsf{Poset}}
\nc{\trno}{T^*\RR^n_{\geq 0}}
\nc{\spectra}{\mathsf{Spectra}}
\nc{\tensorfin}{\tensor^{\fin}}
\nc{\lagptg}{\lag_{pt,pt}^{\cG}}
\nc{\Fin}{\mathcal{F}\mathsf{in}}
\nc{\lagnl}{\lag_{N,\Lambda}}
\nc{\lagmlg}{\lag_{M,\Lambda}^{\cG}}
\nc{\lagsplit}{\lag^{\mathsf{split}}}
\nc{\lagktimes}{(\lag^{\dd k})^\times}
\nc{\lagplanar}{\lag^{\times,\planar}}

\nc{\smsh}{\wedge}
\nc{\un}{\underline}
\nc{\xto}{\xrightarrow}
\nc{\xra}{\xto}
\nc{\tensor}{\otimes}
\nc{\del}{\partial}
\nc{\dd}{\diamond}
\nc{\tri}{\triangle}
\nc{\bb}{\Box}
\nc{\into}{\hookrightarrow}
\nc{\onto}{\twoheadrightarrow}
\nc{\contains}{\supset}
\nc{\transverse}{\pitchfork}
\nc{\uncirc}{\underline{\circ}}
\nc{\Jbar}{\overline{J}}
\nc{\Fbar}{\overline{F}}
\nc{\delbar}{\overline{\del}}
\nc{\thetabar}{\overline{\theta}}
\nc{\omegabar}{\overline{\omega}}

\nc{\colldiff}{\diff^{\del}} 
\nc{\trbar}{\overline{T^*\RR}}
\nc{\tr}{T^*\RR}
\nc{\tsa}{Ts\cA}
\nc{\tsb}{Ts\cB}
\nc{\cmbar}{\overline{\cM}}
\nc{\crbar}{\overline{\cR}}

\nc{\vece}{ {\vec \epsilon}}	
\nc{\vecd}{ {\vec \delta}}
\nc{\ov}{\overline}
\DMO{\op}{op}
\nc{\opp}{ ^{\op}}

\nc{\hiro}{\textcolor{blue}}

\nc{\eqn}{\begin{equation}}
\nc{\eqnn}{\begin{equation}\nonumber}
\nc{\eqnd}{\end{equation}}
\nc{\enum}{\begin{enumerate}}
\nc{\enumd}{\end{enumerate}}

\def\cA{\mathcal A}\def\cB{\mathcal B}\def\cC{\mathcal C}
\def\cF{\mathcal F}\def\cG{\mathcal G}

\def\cM{\mathcal M}
\def\cR{\mathcal R}
\def\cW{\mathcal W}

\def\AA{\mathbb A}\def\CC{\mathbb C}

\def\RR{\mathbb R}

\def\ZZ{\mathbb Z}

\title{Surgery induces exact sequences in Lagrangian cobordisms}
\author{Hiro Lee Tanaka}
\begin{document}

\maketitle

\begin{abstract}
We prove that if $L_0$ and $L_1$ are exact branes intersecting in precisely one point, then there exists a fiber sequence in the $\infty$-category of Lagrangian cobordisms consisting of $L_0$, $L_1$, and a surgery of $L_0$ with $L_1$. By combining this with the exact functor from~\cite{tanaka-exact}, we find analogues of results of Biran and Cornea in the wrapped and exact setting. 
\end{abstract}

\tableofcontents

\section{Introduction}
This paper explores the following hypothesis: Fukaya-categorical information can be detected at the level of Lagrangian cobordisms (in particular, without counting any holomorphic curves). 

For about six years we have known that Lagrangian cobordisms induce exact sequences in the Fukaya category~\cite{biran-cornea, biran-cornea-2, mak-wu, tanaka-exact}. This principle is useful: For instance, by constructing a Lagrangian cobordism induced by Polterovich surgery, one can generalize the Seidel exact sequence~\cite{mak-wu}.

At the same time, a parallel story has shown that the theory of Lagrangian cobordisms on its own (with no regards to Floer theory) has rich algebraic structures---in particular, for any Liouville domain $M$, one can speak of exact sequences in a suitable $\infty$-category of Lagrangian cobordisms. The works~\cite{nadler-tanaka, tanaka-pairing, tanaka-exact}  show that this {\em a priori} rich structure of Lagrangian cobordisms implies many of the known uses of Lagrangian cobordism theory to Floer theory. 

\begin{remark}
This storyline fits a larger narrative: Spectrum-enriched invariants are more powerful than homological (i.e., chain-complex-enriched) invariants. The former has a long history of encoding invariants of smooth topology, so it is natural to seek a spectral enhancement of classical Floer-type invariants---this has led, for example, to a proof of the triangulation conjecture~\cite{manolescu}. 

There are three natural candidates at present for enriching {\em Lagrangian} Floer theory to spectra---microlocal sheaves~\cite{jin-treumann}, deformation-theoretic reformulations~\cite{abouzaid-ucla-talk,lurie-tanaka}, and the theory of Lagrangian cobordisms~\cite{nadler-tanaka}. This paper focuses on the last of these, which has an ``inevitable'' appearance of spectra---while the former two incorporate spectra by changing coefficients, the stable algebraic structure of Lagrangian cobordisms necessitates a spectral enrichment. (Curiously, the stability of Lagrangian cobordisms arises in a {\em distinct} way from stability's appearance in classical cobordism theory---for instance, no Thom spaces are involved in the proof of the spectral enrichment.)
\end{remark}

\begin{remark}
While many statements in symplectic geometry are proven using holomorphic disks, the hypotheses and conclusions often involve ``just the Lagrangian geometry.'' 
It is an open (and vague) question to determine how closely Lagrangian cobordism phenomena parallel the Floer-theoretic arguments.
\end{remark}

This paper continues this storyline. We prove that Polterovich surgery induces exact sequences in the Lagrangian cobordism $\infty$-category $\lag(M)$:

\begin{theorem}\label{theorem.main}
Fix a Liouville domain $M$.
Let $L_0$ and $L_1$ be transverse branes in $M \times T^*\RR^n$ for some $n \geq 0$, and assume $L_0$ and $L_1$ have exactly one intersection point. Then there exists a brane $L_1^\sigma$ (with the same underlying Lagrangian as $L_1$, but with possibly different brane structure) and an exact sequence in $\lag(M)$
	\eqnn
	L_1^\sigma \to L_0 \sharp L_1 \to L_0
	\eqnd 
where the middle term is an exact Polterovich surgery of $L_0$ with $L_1^\sigma$.
\end{theorem}

We emphasize that the above theorem makes no mention of holomorphic curves and requires no use of them in its proof. 

\begin{remark}
If $M$ is Weinstein, Lagrangian coskeleta give a collection of eventually conical branes, and in particular, objects of $\lag(M)$. Following a strategy that we learned in conversations with Ganatra-Pardon-Shende, we anticipate that Theorem~\ref{theorem.main} will allow us to demonstrate that these coskeleta generate $\lag(M)$ as a stable $\infty$-category. This will be the subject of future work.
\end{remark}
 
Now let $\cF(M)$ denote the Fukaya category of those branes geometrically near the skeleton $\sk(M)\subset M$. Also let $\finite(\cF(M)) \subset \fun(\cF(M)^{\op},\chain)$ be the $\infty$-category of finite modules over $\cF(M)$---that is, those contravariant functors that assign every object a finitely generated chain complex (over the base ring $\ZZ$). We have:

\begin{corollary}
The modules represented by $L_0$, $L_1^\sigma$, and $L_0 \sharp L_1$ fit into a short exact sequence in $\finite(\cF(M))$:
	\eqnn
	CF^*(-\times E^n, L_1^\sigma)
	\to CF^*(- \times E^n,
	L_0 \sharp L_1) \to CF^*(-\times E^n,L_0) .
	\eqnd.
\end{corollary}

\begin{proof}
In~\cite{tanaka-pairing}, we showed that there exists a functor 
	\eqnn
	\Xi: \lag(M) \to \finite(\cF(M)),
	\eqnd
taking a brane $L \subset M \times T^*E^n$ to the module given by taking an object $X \in \cF(M)$ and computing a Floer complex $CF^*(X \times E^n, L)$. Because we consider $\cF(M)$ to only consist of those branes geometrically near the skeleton, it follows that the Floer complex is generated by finitely many intersection points, hence is a finite chain complex.  Moreover, the main result of~\cite{tanaka-exact} shows that $\Xi$ is exact---i.e., it preserves exact triangles.
\end{proof}

Theorem~\ref{theorem.main} allows us to deduce results analogous to those of Biran-Cornea~\cite{biran-cornea}. Here is an imprecise formulation for the sake of the introduction:

\begin{theorem}\label{theorem.filtration-vague}
Any cobordism from $L_0$ to $L_1$ with $k$ vertical ends induces a $k$-step filtration on $L_1$ in the Lagrangian cobordism category.
\end{theorem}

See Section~\ref{section.filtrations}. 
By applying $\Xi$ and appealing to the exactness of $\Xi$, we conclude that such a cobordism induces a $k$-step filtration on $L_1$ in $\finite(\cF(M))$. We discuss some conjectural uses of this theorem in Section~\ref{section.conjectures}.

\begin{remark}
To help orient the reader, we remark that the ``Lagrangian cobordism categories'' considered here are different from those in~\cite{biran-cornea, biran-cornea-2}. The geometric differences (which are not the major differences) include the following points:
\enum
\item In this work, we work with Liouville domains $M$, while~\cite{biran-cornea, biran-cornea-2} utilize compact monotone $M$.
\item Here, we work with possibly non-compact but exact Lagrangian branes, while~\cite{biran-cornea,biran-cornea-2} utilize monotone branes.
\enumd

The more substantive differences are in the categorical structures. The $\infty$-category $\lag(M)$ of this paper sees the entire spectrum of Lagrangian cobordisms between any two objects of $M$, and keeps track of all the higher coherences of composition. This allows one to encode, for instance, all the $A_\infty$ structures of endomorphisms, the homotopical data necessary to write down fiber sequences (i.e., exact triangles), and other higher-algebraic data we will need for future applications. In contrast, the category from~\cite{biran-cornea,biran-cornea-2} seems to be equivalent to the (strict) category associated to a planar colored operad: The planar colored operad associated to the s-dot construction of (the monotone version of) the present paper's $\infty$-category. 
\end{remark}

\begin{remark}
To relate the theory of Lagrangian cobordisms to Floer theory in the monotone setting, one needs to fit the framework of~\cite{ritter-smith} to the works~\cite{tanaka-pairing, tanaka-exact}. But insofar as one is studying Lagrangian cobordisms without reference to Floer theory, the adaptation is straightforward, and the monotone analogues of the results in~\cite{nadler-tanaka} all hold.
\end{remark}

\subsection{Notation and conventions}

\begin{notation}
As usual, if $B \subset A$, we let $A \setminus B$ denote the complement of $B$ in $A$. Sometimes $B$ will not be a subset of $A$, in which case we will write $A \setminus B$ as shorthand for $A \setminus (A \cap B)$. 
\end{notation}

\begin{convention}\label{convention.grading}
When discussing gradings, we will choose a covering map $\RR \to S^1$. We take this map to be the homomorphism $t \mapsto \exp(2 \pi i t)$, so that its kernel is $\ZZ \subset \RR$. 
\end{convention}

\subsection{Acknowledgments}
The writing of this work began at the March 2018 AIM workshop on arboreal singularities. We thank AIM for its hospitality and thank the participants for a stimulating environment.

\section{Geometric Preliminaries}

We recall some background on Liouville domains and their branes. 

\subsection{Completions of (possibly non-compact) Liouville domains}\label{section.liouville-completions}

\begin{assumption}\label{assumption.M}
Throughout, we assume that $M$ is the completion of a (possibly non-compact) Liouville domain. We also assume that $2c_1(TM) = 0$. (See Definition~\ref{defn.grading}.)
\end{assumption}

For the reader's convenience, and to set notation, we recall what it means to be a completion of a Liouville domain:

\begin{notation}
\enum
\item $M$ is equipped with a 1-form $\theta$, called the Liouville form.
\item $\omega := d\theta=$ is a symplectic form on $M$.
\item Let $X_\theta$ denote the Liouville vector field associated to $\theta$---that is, the unique vector field satisfying
	\eqnn
	\theta = \iota_{X_\theta}\omega.
	\eqnd
	The condition that $M$ be the completion of a Liouville domain means the following: $M$ can be written as the union of two sets,
		\eqnn
			M = M^0 \bigcup_{\del M^0} \del M^0 \times \RR_{\geq 0}
		\eqnd
	where $M^0$ is a manifold with boundary $\del M^0$, $X_\theta$ points outward along $\del M^0$, and $\del M^0 \times \RR_{\geq 0}$ has $\RR_{\geq 0}$-coordinate parametrized by the flow of $X_\theta$. (In particular, the flow of $X_\theta$ exists for all positive time.)
\item We assume that the flow of $X_\theta$ also exists for all negative time, and that for any $x \in M$, the limit of the time $t$ flow of $x$ exists as $t \to -\infty$.
\enumd
\end{notation}

\begin{defn}
Let $\phi_{X_\theta,t}$ denote the time $t$ flow of the vector field $X_\theta$. We let the {\em skeleton of $M$} denote the set
	\eqnn
	\sk(M) := \bigcap_{t <0 } \phi_{X_\theta,t}(M^0).
	\eqnd
\end{defn}

\begin{caution}
$M^0$ and $\del M^0$ need not be compact. This is why we emphasize that $M$ is the completion of a ``(possibly non-compact)'' Liouville domain. (Usually, Liouville domains are taken to be the data of a compact $M^0$.) However, even in this paper, we typically only consider examples where this non-compactness is captured by the non-compactness of the skeleton, and where the non-compactness of the skeleton is in turn highly controlled---for example, when $M = T^*\RR$. Even in the most general settings, we anticipate that the non-compact skeleta we consider can be modelled as stratified spaces obtained as follows: One begins with a {\em compact} stratified space with boundary, and attaches an infinite cylinder to the boundary.
\end{caution}

\begin{remark}
We refer the reader to~\cite{abouzaid-seidel}, where one version of the wrapped Fukaya category is constructed given an $(M,\theta)$ with $M^0$ compact. In~\cite{ritter-smith}, a wrapped Fukaya category is constructed without the exactness assumption (i.e., without assuming the existence of a $\theta$ defined on all of $M$).
\end{remark}

\begin{example}
If $M = \RR^0$ is a point, $\theta = 0$ renders $M$ the completion of a Liouville domain.
\end{example}

\begin{example}
If $M = T^*Q$ is the cotangent bundle of any smooth manifold $Q$, the usual Liouville form renders $M$ the completion of a Liouville domain, where $M^0$ can be taken to be a relatively compact neighborhood of the zero section. 
\end{example}

\begin{example}
As a case of the above example, let $M = T^*\RR$ be the cotangent bundle of $\RR$ with $\theta = pdq$. Though it is common instead to take $\theta = 1/2(pdq - q dp)$ (which has rotational symmetry), we prefer $pdq$ for the translational symmetry, and to preserve the functoriality of the assignment $Q \mapsto T^*Q$. Note that in this example, $M^0$ is non-compact, as is the skeleton.
\end{example}

\begin{example}
Let $(M_1,\theta_1)$ and $(M_2,\theta_2)$ be completions of Liouville domains. Then their product $M_1 \times M_2$, with the direct sum 1-form $\theta =\theta_1 \oplus \theta_2$, is another completion of a Liouville domain. Because we do not insist on a fixed choice of $M_1^0$ or $M_2^0$, there is no need to ``smooth the corner'' of a product $M_1^0 \times M_2^0$. 
\end{example}

\subsection{Eventually conical submanifolds}
Because $M$ is non-compact, we desire control over the non-compactness of our Lagrangian submanifolds.

\begin{defn}
We say a submanifold $L \subset M$ is {\em eventually conical} if the following is satisfied:

$L$ can be written as a union $L^0 \bigcup_{\del L^0} (\del L^0 \times \RR_{\geq 0})$, where $L^0$ is a compact manifold, possibly with boundary $\del L^0$, and $\del L^0 \times \RR_{\geq 0}$ has $\RR_{\geq 0}$ parametrized by the flow by $X_\theta$. 

(Put another way, outside a compact set, $L^0$ is invariant under the positive flow of $X_\theta$.)
\end{defn}

\begin{example}
If $M = T^*Q$ is a cotangent bundle with the usual Liouville form, any conormal to a submanifold $A \subset Q$ is a conical Lagrangian, hence eventually conical. 
\end{example}

\begin{example}
More generally, one can take a compactly supported Hamiltonian isotopy of any eventually conical Lagrangian, and the result is still eventually conical.
\end{example}

\begin{example}
Any compact submanifold of $M$ is eventually conical. (One takes $L = L^0$ and the boundary $\del L^0$ is empty.)
\end{example}

\subsection{Brane structures}

In Floer theory, various geometric decorations allow one to construct Floer complexes of different flavors. In this paper, we require ourselves to endow Lagrangians with the following structures:
\enum
\item A grading $\alpha:L \to \RR$ (Definition~\ref{defn.grading})
\item A $Pin$-structure (Definition~\ref{defn.pin-structure}), and
\item A primitive $f: L \to \RR$ for the Liouville form $\theta|_L$ (Definition~\ref{defn.primitive}).
\enumd

\begin{defn}\label{defn.brane}
A {\em brane} is the data of an eventually conical Lagrangian $L$ equipped with the above three pieces of data. We will often abuse notation and refer to the brane as $L$, making the brane structure implicit in the notation and surrounding language.
\end{defn}

\begin{remark}
In Floer theory, the above three structures have the following consequences: Exactness permits us to ignore Novikov variables and admit convexity arguments to ensure Gromov compactness; a grading allows us to $\ZZ$-grade Floer complexes; and a Pin structure allows us to use $\ZZ$-linear coefficient rings. 
\end{remark}

\begin{remark}
In the theory of Lagrangian cobordisms, one need not specify brane structures on objects or morphisms to obtain a stable $\infty$-category; however, in parallel to usual cobordims theory, there are different Lagrangian cobordism $\infty$-categories one can construct by demanding that objects and morphisms are equipped with brane structures. 

In this paper we pay close attention to gradings and $Pin$ structures simply because the $\infty$-category associated to these brane structures admits a functor to (finite modules over) a Fukaya category whose objects also have these brane structures~\cite{tanaka-pairing}.
\end{remark}

\begin{notation}\label{notation.GrLag}
Let $\GrLag$ denote the stable Lagrangian Grassmannian; informally, $\GrLag$ is the parameter space for linear Lagrangian subspaces in $\CC^\infty$. Formally, it is defined as the colimit space of the sequential diagram
	\eqnn
	U_0/O_0 \to U_1 / O_1 \to \ldots
	\eqnd
where $U_n$ is the unitary group, $O_n$ is the orthogonal group, and $U_n/O_n$ is the quotient space. The map $U_n/O_n \to U_{n+1}/O_{n+1}$ is obtained by the map
	\eqnn
	[A] \mapsto
	\left[
	\begin{array}{cc}
	1 & 0 \\
	0 & A
	\end{array} 
	\right]
	\eqnd
which, in terms of Lagrangian subspaces, sends $V \subset \CC^n$ to the Lagrangian $\RR \oplus V \subset \CC \times \CC^n \cong \CC^{n+1}$.
\end{notation}

\begin{notation}\label{notation.GrLag_M}
Fix $M$ a symplectic manifold of dimension $2n$. A choice of compatible almost-complex structure renders $TM$ a complex vector bundle, hence induces a map $M \to BU_n$. Then the fiber bundle of Lagrangian Grassmannians over $M$ is obtained by pulling back the map $BO_n \to BU_n$ along the map $M \to BU_n$. We call the bundle $\GrLag_M$. 
\end{notation}

\begin{remark}
Since the space of almost complex structures compatible with $\omega$ is contractible, once one fixes a symplectic form $\omega$, the map $M \to BU_n$ is unambiguously defined up to contractible choice.
\end{remark}

\begin{remark}\label{remark.L-tangent-classifier}
Given a Lagrangian $L \subset M$, one has a commutative diagram
	\eqnn
	\xymatrix{
	& \GrLag_M \ar[r] & O_n \ar[d] \\
	L \ar[r] \ar[ur] & M \ar[r]^{TM} & U_n
	}
	\eqnd
where the map $L \to \GrLag_M$ is the Lagrangian Gauss map.
Moreover, the composite map $L \to \GrLag_M \to O_n$ is precisely the map classifying the tangent bundle of $L$. If one stabilizes $M$ to $M \times T^*\RR$, and $L$ to $L \times \RR^\vee$ (where $\RR^\vee$ is a cotangent fiber), one obtains an obvious, analogous diagram.
\end{remark}

\begin{remark}
Though we have phrased the fact that the composition $L \to \GrLag_M \to O_n$ classifies the tangent bundle $TL$ as a {\em property}, for future purposes, one should note that a brane structure is actually the additional data specifying a homotopy between the composition and the classifying map $L \to O_n$.
\end{remark}

\begin{remark}
We have assumed that $2c_1(TM) = 0$. (See Assumption~\ref{assumption.M}.) In particular, every Lagrangian $L \subset M$ is equipped with a phase-squared map $L \to S^1$ which can locally be expressed as the composition
	\eqnn
	V \xra{TV} \GrLag_M|_V \cong V \times U_n/O_n\xra{\det^2} S^1.
	\eqnd
Here, $|_V$ indicates the restriction of $\GrLag_M$ to a subset $V \subset L$. That this locally defined function can be turned into a global function to $S^1$ is a consequence of the vanishing of $2c_1(TM)$.
\end{remark}

\begin{defn}\label{defn.grading}
A {\em grading} on $L$ is a lift $\alpha: L \to \RR$ of the phase-squared map.
\end{defn}

\begin{remark}
In this paper, the map $\RR \to S^1$ is taken to be the map whose kernel consists of $\ZZ \subset \RR$. See Convention~\ref{convention.grading}.
\end{remark}

\begin{remark}\label{remark.grading-torsor}
The obstruction to such a lift is given by the element in $H^1(L;\ZZ)$ classified by the map $L \to S^1 \simeq K(\ZZ,1)$. If this obstruction vanishes, the collection of all lifts is a torsor over the group $H^0(L;\ZZ)$ of locally constant, integer-valued functions.
\end{remark}

Finally, recall that there is a Lie group extension of the orthogonal group $O_n$, classified by the Stiefel-Whitney class $w_2 \in H^2(BO;\ZZ/2\ZZ)$. This extension will be called $Pin_n$ following the book of Seidel~\cite{seidel-book}. The map $\pin_n \to O_n$ results in a map of classifying spaces $B\pin_n \to BO_n$. Finally, recall that any smooth $n$-manifold has a canonical map to $BO_n$ classifying the tangent bundle.

\begin{defn}\label{defn.pin-structure}
A $Pin$-structure on a Lagrangian $L$ is a lift of the map $L \to BO_n$ to $B\pin_n$. We consider two $Pin$-structures to be equivalent if they are homotopic lifts (i.e., homotopic rel the fixed map to $BO_n$).
\end{defn}

\begin{remark}\label{remark.pin-torsor}
The obstruction to the existence of a $Pin$-structure is classified by $w_2(TL) \in H^2(L;\ZZ/2\ZZ)$. The space of possible $Pin$-structures is a torsor over $H^1(L;\ZZ/2\ZZ)$---concretely, the latter classifies real line bundles over $L$, and tensoring with a line bundle on $L$ changes a given $Pin$-structure on $TL$.
\end{remark}

\begin{notation}[$L^\sigma$]\label{notation.L-sigma}
Let $L$ be a Lagrangian equipped with a fixed grading and a fixed $Pin$-structure. Let $\sigma \in H^0(L;\ZZ) \oplus H^1(L;\ZZ/2\ZZ)$. We let $L^\sigma$ denote the same Lagrangian $L$, but equipped with the grading and $Pin$-structure induced by $\sigma$. (Here we utilize the torsor structure from Remarks~\ref{remark.pin-torsor} and~\ref{remark.grading-torsor}.)
\end{notation}

\begin{remark}\label{remark.classifying-grassmannian}
The obstruction to admitting a brane structure is classified by a map
	\eqnn
	L \to K(\ZZ/2\ZZ,2) \times K(\ZZ,1).
	\eqnd
Moreover, one can show that this classifying map factors through $\GrLag_M$. Thus to show that $L$ admits a brane structure, it suffices to show that the natural map classifying the Lagrangian tangent bundle of $L$ admits a lift
	\eqnn
	\xymatrix{
	& \widetilde{\GrLag_M} \ar[d] \\
	L \ar[r] \ar@{-->}[ur] & \GrLag_M
	}
	\eqnd
where $\widetilde{\GrLag_M} \to \GrLag_M$ is a fibration with fiber $\RR P^\infty \times \ZZ$.  See Remark~\ref{remark.brane-structures} for more generalities on brane structures.
\end{remark}

\begin{defn}\label{defn.primitive}
A {\em primitive} for a Lagrangian $L$ is a choice of smooth function $f: L \to \RR$ such that $df = \theta|_L$.
\end{defn} 

\begin{remark}\label{remark.primitives-along-boundary-vanish}
$f$ is necessarily locally constant along $\del L^0$. However, note that if $\del L^0$ has multiple connected components, one cannot always choose $f$ to equal zero along $\del L^0$. While the assumption that one can choose $f|_{\del L^0} = 0$ was utilized in~\cite{abouzaid-seidel} to obtain Gromov compactness, one does not need this assumption~\cite{gao-wrapped-floer}.
\end{remark}

\subsection{Eventually conical branes in products}
It is a well-known annoyance that products of eventually conical subsets need not be eventually conical. Here we present a workaround definition; the idea is to {\em remember} the data of the two projection maps $M_1 \times M_2 \to M_i$. Similar workarounds can be found in~\cite{gao-wrapped-floer}.

\begin{notation}\label{notation.product-domains}
Let $M_i$ be completions of Liouville domains with Liouville forms $\theta_i$, $i=1,2$. Then we consider the product $M_1 \times M_2$ to be a completion of a Liouville domain by endowing it with the form $\theta_1 \oplus \theta_2$. Explicitly, this 1-form acts by sending a tangent vector $(u_1,u_2) \in TM_1 \times TM_2 \cong T(M_1 \times M_2)$ to the number $\theta(u_1) + \theta(u_2)$. 
\end{notation}

\begin{defn}\label{defn.conical-in-product}
Let $M_i$ be a completion of a Liouville domain for $i=1,2$. Then a subset $L \subset M_1 \times M_2$ is said to be eventually conical {\em in the product} if one of the following holds:
\enum
\item $L$ is eventually conical in $M_1 \times M_2$ (with respect to the Liouville form $\theta_1 \oplus \theta_2$), or
\item Outside of a compact set, $L$ can be written as a finite disjoint union of products
	\eqn\label{eqn.conical-in-product}
	\coprod_i L_1^{(i)} \times L_2^{(i)}
	\eqnd
where for every $i$, $L_1^{(i)} \subset M_1$ is an eventually conical subset, as is $L_2^{(i)} \subset M_2$.
\enumd
\end{defn}

\begin{example}
If $L_1$ and $L_2$ are eventually conical in $M_1$ and $M_2$, respectively, then $L_1 \times L_2 \subset M_1 \times M_2$ is eventually conical in the product.
\end{example}

\begin{remark}\label{remark.product-gromov-compactness}
If $L$ and $L'$ are eventually conical in the product $M_1 \times M_2$, then one can define wrapped Floer theory for $L$ and $L'$. The only point to check is that Gromov compactness holds. This follows because if one is given a tuple of Hamiltonian chords, one can place uniform energy estimates on any disk limiting to these chords via constants that only depend on a finite number of parameters---a constant associated to a compact region, and a constant associated to each connected component $i$ in Equation~\eqref{eqn.conical-in-product}. See for instance~\cite{gao-wrapped-floer}.
\end{remark}

\subsection{Vertically collared branes}

Let $M$ be the completion of a Liouville domain. (See Section~\ref{section.liouville-completions}.) 

\begin{defn}\label{defn.vertically-collared}
A Lagrangian $L \subset T^*\RR \times M$ is called {\em vertically collared} if the following holds: There exists some compact subset $K \subset T^*\RR$, a collection of curves $\gamma_i \subset T^*\RR \setminus K$ that are eventually conical, and a collection of Lagrangians $X_i \subset M$ that are eventually conical, such that
	\eqn\label{eqn.vertically-collared}
	L\setminus (K \times M)
	=
	\coprod_i \gamma_i \times X_i.
	\eqnd
See Figure~\ref{figure.vertically-collared}. We further say that a {\em brane} $L \subset T^*\RR \times m$ is vertically collared if the equality above is an equality of branes.
\end{defn}

\begin{remark}
Let us explain what we mean by an equality of branes. We endow each $\gamma_i$ with the grading such that, where $\gamma_i$ is collared with $p>>0$, $\alpha_{\gamma_i} = -1/2$. (This is compatible with the convention from Section~\ref{section.stabilization}.) Given two branes, their direct product inherits an induced brane structure; this explains the brane structure we equate in~\eqref{eqn.vertically-collared}.  
\end{remark}

\begin{figure}
		\[
			\xy
			\xyimport(8,8)(0,0){\includegraphics[width=1in]{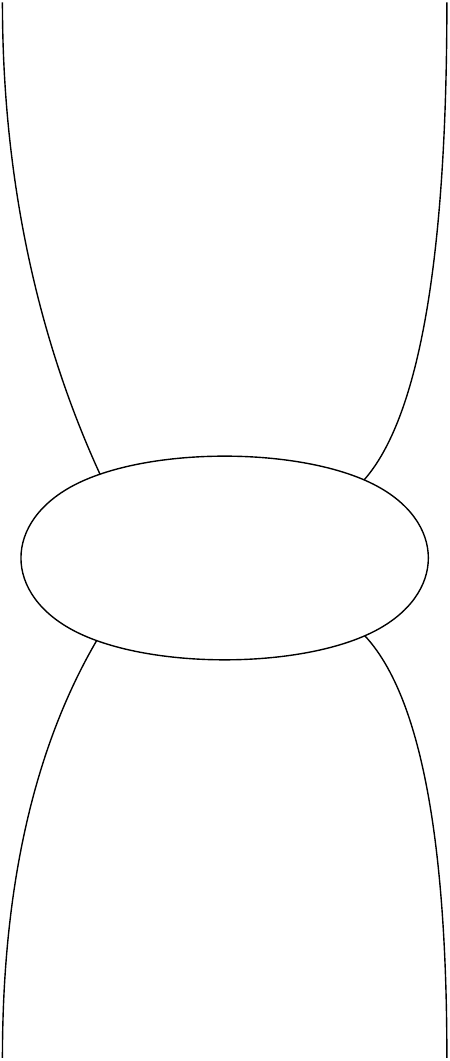}}
			,(-0.8,7)*+{X_{n}}
			,(4,7)*+{\ldots}
			,(8.8,7)*+{X_1}
			,(-0.8,0.4)*+{X_1'}
			,(4,0.4)*+{\ldots}
			,(9,0.4)*+{X_{n'}'}
			\endxy
		\]
\caption{
The projection of a vertically collared Lagrangian $L \subset T^*\RR \times M$ to $T^*\RR$. See~\eqref{eqn.vertically-collared-prime} for an explanation of the labels $X_i$ and $X_j'$.
}\label{figure.vertically-collared}
\end{figure}

\section{Lagrangian cobordism preliminaries}
Here we recall some basic facts about the Lagrangian cobordism $\infty$-category $\lag(M)$ of~\cite{nadler-tanaka}. All facts are given without proofs, which can be found in the papers~\cite{nadler-tanaka,tanaka-exact,tanaka-pairing}.

\subsection{Stabilization}\label{section.stabilization}

Given a brane $L \subset M$, there is a natural way to produce another brane in $T^*E \times M$, where $E = \RR$ is the real line. Namely, let $l = T^*_0 E$ denote the cotangent fiber at $0 \in E$, endowed with the grading $\alpha_l = -1/2$. Then there is an induced brane structure on $l \times L \subset T^*E \times M$: One sets $\alpha_{l \times L}((0,p),x) = \alpha(x) - 1/2.$

\begin{notation}
We will denote the brane $l \times L$ by $L^\dd$.
\end{notation}

\subsection{Sketch definition of $\lag(M)$}
Let $\lag^{\dd 0}(M)$ denote the $\infty$-category (i.e., quasi-category) whose objects are eventually conical branes inside $M$, and whose 1-simplices are branes $Q \subset T^*\RR \times M$ which avoid the skeleton of $M$ in the negative cotangent direction fo $T^*\RR$, and which are collared outside a compact interval of the zero section. The $k$-simplices for $k \geq 2$ are higher Lagrangian cobordisms also avoiding the skeleton appropriately. For details, we refer to~\cite{nadler-tanaka}, in terms of the notation there, we set $\Lambda = \sk(M)$ in this paper.

The stabilization procedure yields maps of $\infty$-categories
	\eqnn
	\lag^{\dd 0}(M) \to \lag^{\dd 0}(T^*\RR \times M) \to \ldots
	\to \lag^{\dd 0}(T^*\RR^n \times M) \to \ldots
	\eqnd
and we let $\lag(M)$ denote the union of this sequence.

\begin{remark}\label{remark.stabilization-equivalence}
By construction, a brane $L \subset T^*\RR^n \times M$ is the same object as $L^{\dd} \subset T^*\RR^{n+1} \times M$ in the $\infty$-category $\lag(M)$.
\end{remark}

\subsection{Eventually linear Hamiltonian isotopies induce equivalences}
Given any time-dependent Hamiltonian $H_t$ on $M$ and a brane $L \subset M$, one can construct a Lagrangian cobordism from $L$ to its flow by the Hamiltonian so long as $H_t = 0$ for $|t|>>0$. We say that $H_t$ is {\em eventually linear} if for every time $t$, and outside a relatively compact neighborhood of $\sk(M)$, we have that $X_\theta(H_t) = H_t$. (The ``eventually'' refers to distance away from the skeleton, and not to the time coordinate.)

\begin{prop}\label{prop.linear-isotopy-equivalence}
Let $H_t: M \to \RR$ be an eventually linear, time-dependent Hamiltonian isotopy which is compactly supported in time. Then $H$ induces an equivalence in $\lag(M)$ from a brane $L \subset M$ to its image under the Hamiltonian isotopy induced by $H_t$.
\end{prop}

We refer the reader to~\cite{tanaka-exact} for details.

\subsection{$\Lambda$-avoiding objects are zero objects}
We also have the following basic property:

\begin{prop}\label{prop.lambda-avoiding-zero-objects}
Let $L \subset T^*\RR^n \times M$ be an object of $\lag(M)$. If there is a neighborhood $U$ of $\RR^n \times \sk(M)$ such that $L \cap U = \emptyset$, then $L$ is a zero object.
\end{prop}

We refer to~\cite{nadler-tanaka} for the proof.

\subsection{Mapping cones and mapping kernels}\label{section.mapping-cone}
Fix a morphism $Q: L_0 \to L_1$ in an arbitrary $\infty$-category. Recall that a {\em cofiber}, or {\em mapping cone} for $Q$ is a pushout diagram
	\eqnn
	\xymatrix{
	L_0 \ar[r]^Q \ar[d] & L_1 \ar[d] \\
	0 \ar[r] & \cone(Q)
	}
	\eqnd
where $0$ is a zero object. Likewise, a {\em fiber}, or {\em (homotopy) kernel} of $Q$ is a pullback diagram
	\eqnn
	\xymatrix{
	\ker(Q) \ar[d] \ar[r] & L_0 \ar[d]^Q \\
	0 \ar[r] & L_1.
	}
	\eqnd
The main result of~\cite{nadler-tanaka} depends on the following:

\begin{theorem}\label{theorem.mapping-cone}
Any morphism $Q: L_0 \to L_1$ in $\lag(M)$ admits a mapping cone and a homotopy kernel. 
\end{theorem}

\begin{remark}
Moreover, $\kernel(Q)$ can be modeled informally as follows: If $Q$ is collared by $L_0$ along $(-\infty,-T] \subset \RR$ and by $L_1$ along $[T,\infty) \subset \RR$, isotope $(-\infty,-T]$ into a curve beginning at $-T \in \RR$ and eventually becoming conical in the negative cotangent direction; do the same for $[T,\infty)$. We refer the reader to~\cite{nadler-tanaka} for details, and to Figure~\ref{figure.LL3} for an image of the kernel. 
\end{remark}

\begin{remark}
By formal non-sense for stable $\infty$-categories, the kernel is obtained from the mapping cone by a loop functor. (And the mapping cone is obtained from the kernel by a suspension fucntor.) In the setting of Lagrangian cobordisms, these functors do not change the underlying Lagrangian, but change the brane structure by the obvious shift.
\end{remark}

\begin{example}
In $\lag(M)$, a zero object is given by the empty Lagrangian. Take a curve $\Gamma$ as in Notation~\ref{notation.Gamma} below, and consider $\Gamma \times L \subset T^*\RR \times M$ as a morphism from $\emptyset$ to $L$. We call the kernel of this map $\Omega L$ following the conventions from stable homotopy theory. $\Omega L$ is easily computed to be Hamiltonian isotopic to $L^{\dd}$, but with grading shifted by $+1$. That is, after a suitable Hamiltonian isotopy, one can compute that 
	\eqnn
	\alpha_{\Omega L} = \alpha_{L^{\dd}} + 1.
	\eqnd
Dually, if computing the mapping cone $\Sigma L$ of a map $L \to \emptyset$, we find
	\eqnn
	\alpha_{\Sigma L} = \alpha_{L^{\dd}} - 1.
	\eqnd
\end{example}

Finally, we will use the following standard categorical fact:

\begin{prop}\label{prop.cone-zero-equivalence}
Let $Q: L_0 \to L_1$ be a morphism in any stable $\infty$-category with a zero object. Assume that the pushout
	\eqnn
	\xymatrix{
	L_0  \ar[d] \ar[r]^Q & L_1 \ar[d] \\
	0 \ar[r] & \cone(Q)
	}
	\eqnd
exists, and that $\cone(Q)$ is a zero object. Then $Q$ is an equivalence. 
\end{prop}

\begin{proof}
If $\cone(Q)$ is a zero object, the bottom horizontal arrow is an equivalence. Since the $\infty$-category is stable, the square is also a pullback, and the pullback of an equivalence is an equivalence.
\end{proof}

\subsection{Disjoint union is coproduct}

Finally, given two branes $L_0, L_1$ in $M \times T^*\RR^n$, their categorical coproduct can be realized geometrically:

\begin{prop}\label{prop.disjoint-union-coproduct}
Let $T^*_{q_0} \RR \times L_0$ and $T^*_{q_1} \RR \times L_1$ be two stabilizations of $L_0$ and $L_1$, respectively. If $q_0 \neq q_1$, then the disjoint union brane
	\eqnn
	T^*_{q_0} \RR \times L_0
	\coprod
	T^*_{q_1} \RR \times L_1
	\eqnd
is a categorical coproduct $L_0 \coprod L_1$ in $\lag(M)$.
\end{prop}

\section{Polterovich surgery}

\begin{remark}
The surgery of immersed Lagrangian submanifolds was first introduced by Polterovich in~\cite{polterovich-surgery} following the construction of surgered Lagrangians in dimension 2 by Lolande-Sikarov~\cite{}. In~\cite{polterovich-surgery}, Polterovich also mentions that Audin had studied Lagrangian surgery in the context of Lagrangian cobordism theory.
\end{remark}

\begin{remark}
For the reader's convenience, we carefully keep track of the structures one must use to decorate surgeries of branes (as opposed to surgeries of undecorated Lagrangians). The exposition is meant to be accessible to graduate students and to mathematicians unfamiliar with details of common symplectic arguments.
\end{remark}

\subsection{The construction}\label{section.surgery}
The familiar reader may wish to skip this section.

Let $M = (M,\omega)$ be a symplectic manifold of real dimension $2n$ and fix two Lagrangians $L_0, L_1 \subset M$. Assume that $p \in L_0 \cap L_1$ is a transverse intersection point. We review the construction of a (possibly immersed) Lagrangian $L_0 \sharp_p L_1$, which we call a {\em surgery} of $L_0$ with $L_1$ along $p$. (Indeed, there are at least two natural surgeries one can construct: See Remark~\ref{remark.surgery-reverse}.) The main definition is Definition~\ref{defn.lag-surgery}.

\begin{notation}
We let $\RR^n \subset \CC^n$ be the standard set of purely real vectors, and $i\RR^n \subset \CC^n$ the set of purely imaginary vectors. Likewise, we let $S^{n-1} \subset \RR^n$ denote the unit sphere, and $iS^{n-1} \subset i \RR^n$ denote the purely imaginary unit sphere.
\end{notation}

We first recall a basic fact:

\begin{prop}\label{prop.local-model}
In a neighborhood of $p$, we may model $L_0$ and $L_1$ as the standard copies of $\RR^n$ and $i\RR^n$ in $\CC^n$.
\end{prop} 

\begin{proof}
For a small neighborhood $U \subset M$ containing $p$, choose a Darboux chart $\phi: U \to \CC^n$. The local model of the proposition is achieved by noting the following:
\enum
\item There exists a Hamiltonian isotopy deforming $\phi(L_0)$ to the standard copy of $\RR^n = span(x_1,\ldots,x_n)$ in some neighborhood of $\phi(p)$. (For example, by identifying $T_{\phi(p)}\phi(L_0)$ with a linear subspace of $\CC^n$, we observe that $\phi(L_0)$ is locally the graph of a closed 1-form defined on $T_{\phi(p)}\phi(L_0)$. One can choose a Hamiltonian isotopy from the graph of this closed 1-form to the graph of the trivial 1-form.)
\item Likewise, there exists a Hamiltonian isotopy that preserves $\RR^n$ and deforms $\phi(L_1)$ to the standard copy of $i\RR^n = span(y_1,\ldots,y_n)$ in a neighborhood of $p$. (Here we have used that $p$ is a transverse intersection point of $L_0$ and $L_1$.) 
\enumd
By composing $\phi$ with these Hamiltonian isotopies (and passing to a smaller neighborhood inside $U$ as necessary), the proposition is proven.
\end{proof}

\begin{remark}[The local picture preserves primitives at $p$]
Both isotopies can be chosen to fix $\phi(p)$; in particular, if each $L_i$ is equipped with a primitive $f_i: L_i \to \RR$ such that $df_i = \theta|_{L_i}$, these deformations do not change the value of $f_i(p)$. 
\end{remark}

\begin{notation}[$\Gamma$]\label{notation.Gamma}
In what follows, we fix a smooth curve $\Gamma \subset \CC$ with the following properties:
\enum
\item $\Gamma$ is contained in the first quadrant, so $(x,y) \in \Gamma \implies x \geq 0 \land y \geq 0$, 
\item There exists a real number $A > 0$ such that outside the open box 
	\eqnn
	\AA = (0,A) \times i(0,A) \subset \CC,
	\eqnd 
$\Gamma$ is equal to a union of the two rays
	\eqnn
	[A, \infty)  \coprod i[A, \infty) \subset \RR \cup i\RR \subset \CC,
	\eqnd
and
\item The two projection maps $\Gamma \cap \AA \to \RR$ and $\Gamma \cap \AA \to i\RR$ are both diffeomorphisms onto their image.
\enumd 
See Figure~\ref{figure.Gamma-curve}.
\end{notation}

\begin{figure}[h]
		\[
			\xy
			\xyimport(8,8)(0,0){\includegraphics[width=2in]{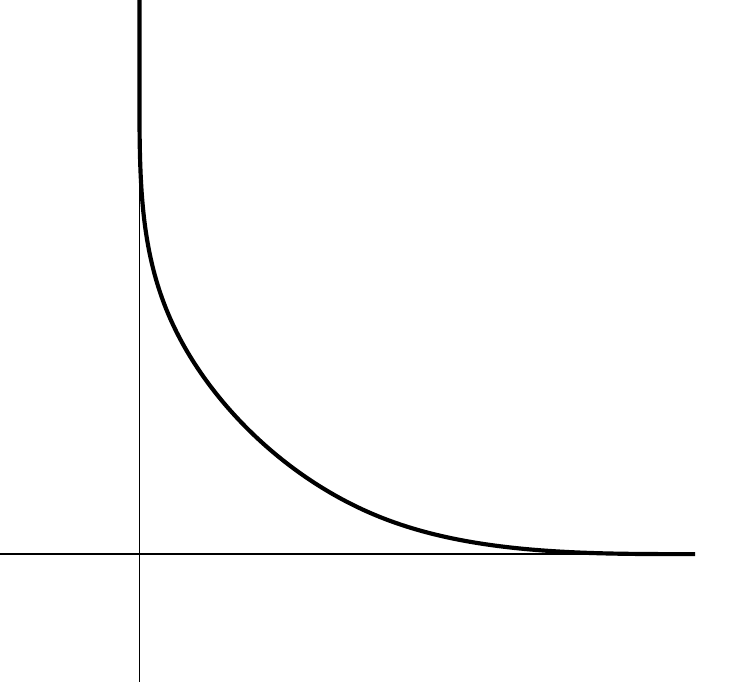}}
			,(5,4)*+{\Gamma}
			\endxy
		\]
\caption{
A choice of $\Gamma \subset \CC \cong T^*\RR$.
}
\label{figure.Gamma-curve}
\end{figure}

\begin{notation}
Let $S^{n-1} \subset \RR^{n} \subset \CC^{n}$ be the unit sphere in $\RR^n$. We let $\Gamma\cdot S^{n-1} \subset \CC^{n}$ denote the set
	\eqnn
	\Gamma\cdot S^{n-1}
	:= 
	\{
	(z x_1,\ldots,z x_n)
	\, | \,
	z \in \Gamma 
	\land
	(x_1,\ldots,x_n) \in S^{n-1}
	\}.
	\eqnd
\end{notation}

\begin{prop}\label{prop.local-surgery}
$\Gamma\cdot S^{n-1} \subset \CC^{n}$ is a Lagrangian submanifold diffeomorphic to $\RR \times S^{n-1}$. Outside the rectangular solid $[-A, A]^{2n} \subset \RR^{2n} \cong \CC^n$, the set $\Gamma\cdot S^{n-1}$ is equal to $\RR^n \cup i \RR^n \setminus [-A, A]^{2n}$.
\end{prop}

\begin{proof}
Choosing a parametrization $\gamma: \RR \to \Gamma$, we have an induced map $\RR \times S^{n-1}\to \CC^{n}$. It is easy to see that this is a diffeomorphism onto its image, so that $\Gamma\cdot S^{n-1}$ is a submanifold of $\CC^{n}$ diffeomorphic to $\RR \times S^{n-1}$. 

Let us write $\gamma(t) = \gamma_1(t) + i \gamma_2(t) \in \CC$. We must show that the embedding
	\eqnn
	j: \RR \times S^{n-1} \to \CC^{n},
	\qquad
	(t,x) \mapsto \gamma(t)\cdot x
	\eqnd
is a Lagrangian embedding with respect to the standard symplectic structure $\sum dx_i dy_i$. Note that $S^{n-1} \subset \RR^n$ is isotropic and complex multiplication preserves isotropics; so it suffices to check that for any $x = (x_1, \ldots, x_n) \in S^{n-1}$ and for any $u = \sum_{i=1}^n u_i {\frac {\del}{\del x^i}} \in T_x S^{n-1}$, we have that $j^*\omega({\frac {\del}{\del t}}, u) = 0$. This is a straightforward computation:
	\begin{align}
	j^*\omega({\frac {\del}{\del t}}, u)
	&=
	\omega (Dj ({\frac {\del}{\del t}}), Dj(u)) \nonumber \\
	&=
	\omega ({\frac {d \gamma_1}{dt}} \sum_{i=1}^n x_i {\frac {\del}{\del x^i}} + {\frac {d \gamma_2}{dt}} \sum_{i=1}^n x_i{\frac{\del}{\del y^i}}, \gamma_1 u + \gamma_2 J(u)) \nonumber \\
	&= ({\frac {d \gamma_1}{dt}} \gamma_2 - \gamma_1 {\frac {d \gamma_2}{dt}}) \sum_{i=1}^{n} x_i u_i \nonumber \\
	&= 0. \nonumber
	\end{align}
(We have used the fact that any tangent vector $u$ to $x$ is orthogonal to $x$.)
\end{proof}

By Proposition~\ref{prop.local-model}, we know that $L_0$ and $L_1$ inside $M$ can be modeled near $p$ as the standard linear subspaces $\RR^n, i \RR^n$ inside $\CC^n$. Letting $U \subset M$ be such a small neighborhood about $p$ and letting $\phi: U \to \CC^n$ be a local model so that $\phi(L_0 \cap U) = \RR^n \cap \phi(U)$ and $\phi(L_1 \cap U) = i \RR^n \cap \phi(U)$, consider the submanifold
	\eqnn
	L_0 \sharp_p L_1 
	:= 
	 \phi^{-1}\left(
		\Gamma \cdot S^{n-1}
		\right)
	\bigcup
	\left(
	L_0 \cup L_1 \setminus U 
	\right)
	\eqnd
where $A$ is chosen small enough so that $[-A, A]^{2n} \subset \phi(U)$. 

\begin{defn}\label{defn.lag-surgery}
We call $L_0 \sharp_p L_1$ a {\em Lagrangian surgery}, or a {\em Polterovich surgery} of $L_0$ and $L_1$ at $p$. (See Figure~\ref{figure.surgery-summary}.)
\end{defn}

\begin{example}
The Lagrangian surgery of two points is empty. 
\end{example}

\begin{figure}
		\[
			\xy
			\xyimport(8,8)(0,0){\includegraphics[width=4in]{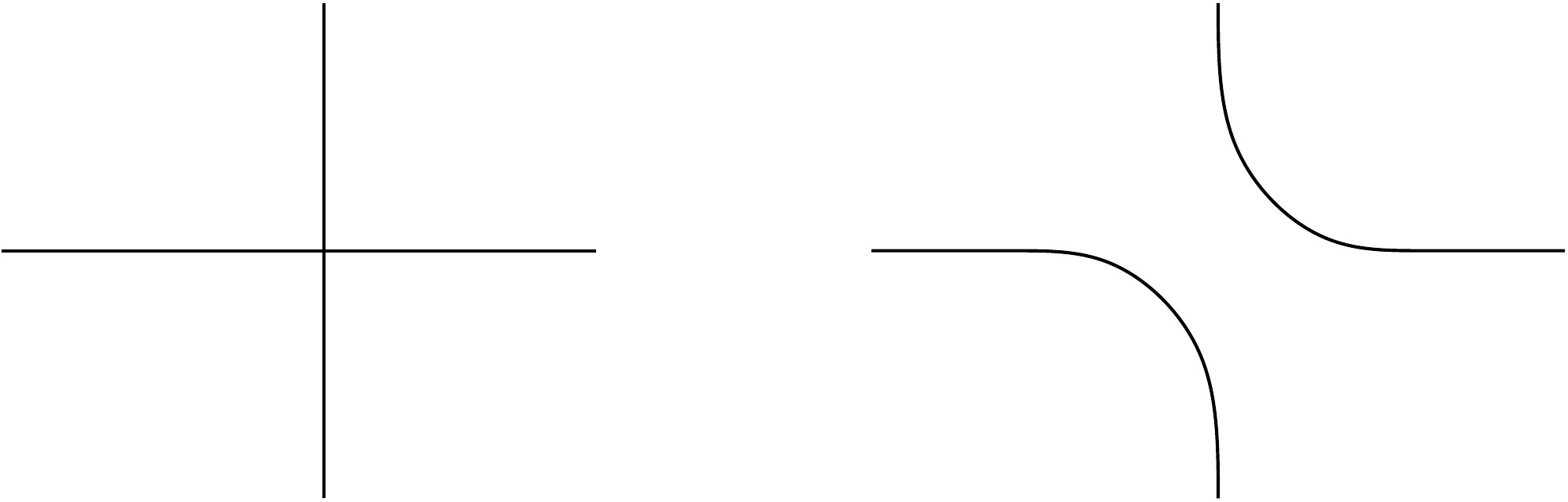}}
			,(1.5,8)*+{L_1}
			,(2.8,3.3)*+{L_0}
			,(7,6.5)*+{L_0 \#_p L_1}
			\endxy
		\]
\caption{A local picture of Lagrangian surgery in the case $n=1$. Note that in the picture $L_0\sharp_p L_1$ is disconnected, but in general (for $n \geq 2$), the surgery is connected.}
\label{figure.surgery-summary}
\end{figure}

\begin{remark}\label{remark.surgery-reverse}
There is another Lagrangian surgery one could obtain---take the complex conjugate $\overline{\Gamma}$, and locally replace $L_0 \cup L_1$ by $\overline{\Gamma} \cdot S^{n-1}$. 
\end{remark}

\begin{remark}
If $L_0 \cap L_1$ consists of exactly one point $p$, then $L_0 \sharp_p L_1$ is an embedded Lagrangian.
\end{remark}

\subsection{Brane structures on Lagrangian surgeries}

\begin{prop}\label{prop.brane-surgery}
Let $L_0, L_1\subset M$ be Lagrangians equipped with a grading and a $Pin$-structure. Assume that $L_0$ and $L_1$ intersect transversally at exactly one point $p$. Then the Lagrangian $L_0 \sharp_p L_1$ admits a grading and a $Pin$-structure, and this can be chosen such that for some element $\sigma \in H^0(L_1;\ZZ) \oplus H^1(L_1;\ZZ/2\ZZ)$, the grading and $Pin$ structures on $L_0 \sharp_p L_1$ agrees with that of $L_0 \cap L_0 \sharp_p L_1$ and of $L_1^\sigma \cap L_0 \sharp_p L_1$. (See Notation~\ref{notation.L-sigma}.)
\end{prop}

\begin{proof}
Let us first show that the local model $\Gamma\cdot S^{n-1} \subset \CC^n$ admits a brane structure. By Remark~\ref{remark.classifying-grassmannian}, it suffices to show that the composite
	\eqnn
	S^{n-1} \simeq S^{n-1} \times \RR \cong \Gamma\cdot S^{n-1} \to \GrLag
	\eqnd
is null-homotopic. This is obvious because the map $S^{n-1} \to \GrLag$ factors through $\RR^n \subset \CC^n$ and the inclusion $S^{n-1} \to \Gamma \cdot S^{n-1}$ is a homotopy equivalence.

Further, the inclusions $S^{n-1} \to \Gamma \cdot S^{n-1}$ and $iS^{n-1} \to \Gamma \cdot S^{n-1}$ realize a homotopy between the maps classifying the brane structure on $S^{n-1}$ and on $i S^{n-1}$. This homotopy is realized by rotating $\RR^n$ into $i \RR^n$ by multiplying the elements of $\RR^n$ along a path of complex numbers starting at 1 and ending at $i$. We conclude that the brane structure on $i S^{n-1}$ extends to one on $i \RR^n$. Let $iD^n \subset i\RR^n$ be the usual disk; we then have a commutative diagram
	\eqnn
	\xymatrix{
	iS^{n-1} \ar[d] \ar[dr] \\
	i D^n \ar[d] \ar[r] & \widetilde{\GrLag_M} \ar[d]\\
	L_1 \ar@{-->}[ur]^{\sigma} \ar[r] & \GrLag_M
	}
	\eqnd
and we must determine whether a lift $\sigma$ exists as indicated. By assumption that $L_1$ admits a brane structure, we know that there does exist a lift of the map $L_1 \to \GrLag_M$ to $\widetilde{\GrLag_M}$; the only question is whether one can choose a lift compatible with the map from $iD^n$. This is clear, as there is no obstruction to extending a given brane structure on a point $\RR^0 \simeq D^n \subset L_1$ to all of $L_1$ once one knows that $L_1$ admits a brane structure.

Thus by shifting the brane structure on $L_1$ by some appropriate group element $\sigma$ to obtain a new brane $L_1^\sigma$, one can ensure that the brane structure of $L_1^\sigma$ restricts to the brane structure on $i D^n$. This completes the proof.
\end{proof}

\begin{remark}[Brane structures on surgeries, in general]\label{remark.brane-structures}
While gradings and Pin structures are utilized to construct $\ZZ$-graded, $\ZZ$-linear Floer theories, a Floer theory that is linear over ring {\em spectra} will require other topological decorations of our Lagrangians. (See, for example, work of Jin-Treumann~\cite{jin-treumann} which suggests the structures needed to produce a microlocal invariant linear over the sphere spectrum.) 

We anticipate that the admissibility of a brane structure can be measured through purely topological invariants of a Lagrangian and its ambient symplectic manifold. Let us make the following hypothesis about what a brane structure is:

{\bf Hypothesis:} To every ring spectrum $R$, there exists a bundle $B_R \to \GrLag_M \to M$ with fiber $Pic(R)$; a brane structure on $L$ with respect to $R$ is a choice of lift of the map $\tau_L: L \to \GrLag_M$ to $B_R$. 

Moreover, let us assume that $B_R \to \GrLag_M$ has the following lifting property: If $L$ admits a lift of $\tau_L$, then for any $x \in L$ and for any point $\tilde x$ in the fiber above $\tau_L(x)$,  there is a lift $L \to B_R$ such that $x$ is sent to $\tilde x$. (This ensures that given any map $D^n \to L$ and a lift of $D^n \to L \to \GrLag_M$ to $B_R$, one can find a compatible lift of $\tau_L$.)

Then the proof of Proposition~\ref{prop.brane-surgery} carries over to show that if $L_0$ and $L_1$ are branes intersecting transversally at exactly one point, their surgery $L_0 \sharp L_1$ supports a brane structure compatible with $L_1^\sigma$ for some $\sigma$.
\end{remark}

\begin{caution}\label{warning.complex-structure-conjugate}
Finally, our conventions in this paper are to endow $\CC \cong T^*\RR$ with the symplectic structure $dpdq$ of $T^*\RR$, and the corresponding compatible almost-complex structure. Thus, under the standard identification $x \mapsto q, y \mapsto p$, we conclude that $J_{T^*\RR}$ is the {\em complex conjugate} of the usual complex structure on $\CC$. (This is a famous incompatibility between $T^*\RR$ and $\CC$.) 
\end{caution}

\begin{example}\label{example.surgery-gauss-map}
Let us describe the map from $\Gamma \cdot S^{n-1} \to \GrLag$ more explicitly. Choose a parametrization $\gamma: \RR \to \Gamma$, and let $v_1,\ldots, v_{n-1} \in T_x S^{n-1}$ be a basis for the tangent space at $x \in S^{n-1}$. Via the standard embedding $T S^{n-1} \into T \RR^n \subset T\CC^n|_{\RR^n} \cong \RR^n \times \CC^n$, one may identify $v_1,\ldots,v_{n-1}$ as elements of $\RR^n \subset \CC^n$. Then the collection
	\eqnn
	{\frac {d \gamma}{dt}} , \gamma(t) v_1,\ldots, \gamma(t) v_{n-1}
	\in \CC^n
	\eqnd
forms a basis for the tangent space of $\Gamma \cdot S^{n-1}$ at a point $\gamma(t) \cdot x$. The composition
	\eqnn
	\Gamma \cdot S^{n-1} \to GL_n(\CC) \to GL_n(\overline{\CC}) \simeq U_n \to U_n/O_n \simeq \GrLag_n \to \GrLag
	\eqnd
is the Gauss map in Remark~\ref{remark.classifying-grassmannian}. 
\end{example}

\begin{example}
Using Example~\ref{example.surgery-gauss-map}, let us describe the grading on the surgery. First, note that $\RR^n \subset \CC^n$ has constant squared-phase given by $1 \in S^1$. Let us grade the Lagrangian $\RR^n$ and choose a lift to $\RR$, say by declaring $\alpha = 0$. Then a simple computation shows that the path $(\gamma(t),0,0,\ldots,0)$ from $(1,0,\ldots,0) \in S^{n-1} \subset \RR^n$ to $(i,0,\ldots,0) \in i S^n \subset i \RR^n$ has squared-phase
	\eqnn
	\overline{
	\left({\frac
	{\dot \gamma(t) (\gamma(t))^n}
	{|\dot \gamma(t) (\gamma(t))^n|}
	}
	\right)^2
	}
	\in S^1
	\eqnd
at time $t$.

 Now extend the grading $\alpha = 0$ on $S^{n-1} \subset \RR^n$ to the grading on $\Gamma \cdot S^{n-1}$. Then, to the point $(i,0,\ldots,0) \in iS^{n-1} \subset \RR^n \sharp i \RR^n$, the extended grading assigns the real number
	\eqnn
	{\frac 1 2} - {\frac {n-1} {2}}.
	\eqnd
(Of course, the squared-phase map on $i\RR^n \subset \CC^n$ is constant, so this is the value of the grading on all of $i \RR^n \cap \RR^n \sharp i \RR^n$.) Note that we are making use of Warning~\ref{warning.complex-structure-conjugate}.
\end{example}

Now we tackle exactness.

\begin{prop}\label{prop.surgery-primitive}
Let $L_0,L_1 \subset M$ be Lagrangians equipped with primitives $f_i: L_i \to \RR$. Let $p \in L_0 \cap L_1$ be a transverse intersection point. Then there exists a constant $C$ (depending only on the behavior of $L_0$ and $L_1$ near $p$) such that if $|f_0(p) - f_1(p)| < C$, then there exists a surgery $L_0 \sharp_p L_1 \subset M$ which is an exact Lagrangian, whose primitive can be chosen to agree with $f_i$ along $\left(L_0 \sharp_p L_1\right) \bigcap L_i$.
\end{prop}

\begin{remark}
If $\dim M \geq 2n$ and one does not care about respecting particular primitives $f_i: L_i \to \RR$, then one does not need the full hypotheses of Proposition~\ref{prop.surgery-cobordism} to see that the surgery admits a primitive; the argument is the same in spirit to the argument in Proposition~\ref{prop.brane-surgery}.

We must be more careful when $n=1$; moreover, even for $n \geq 2$, there may be occasion in the future to demand that a primitive $f_L: L \to \RR$ have the property that $f|_{\del L} = 0$. (See Remark~\ref{remark.primitives-along-boundary-vanish}.) We include Proposition~\ref{prop.surgery-primitive} and its proof for these reasons.
\end{remark}

\begin{proof}
We follow Notation~\ref{notation.Gamma}.
Let us choose $A$ small enough so that we know the union $L_0 \cup L_1$ can be modeled inside $[-2A,2A]^{2n} \subset \CC^n$ as $\RR^n \cup i \RR^n$. We let $\theta$ be a Liouville form pulled back along the Darboux chart. We henceforth identify points of $M$, $L_0$, and $L_1$ with points of $\CC^n, \RR^n$, and $i\RR^n$, respectively. 

Without loss of generality, we assume that $f_1(0) > f_0(0)$. 
By choosing $A$ small enough, we can assume that $f_1(x_1) > f_0(x_0)$ for any $x_1 \in i \RR^n, x_0 \in \RR^n$.

We separate the cases $n=1$ (when the surgery is disconnected near $p$)
and $n \geq 2$ (when the local model is connected).

{\bf ($n=1$).} 
Fix $x_0 \in \RR_{<0}$ and $x_1 \in i \RR_{>0}$, both with norm between $A$ and $2A$. We seek a curve $\Gamma$ as in Notation~\ref{notation.Gamma} such that the region $R$ between the curves $\Gamma$, $\RR$, and $i\RR$ has area given by $f_1(0) - f_0(0)$. Of course, if this difference is small enough, such a $\Gamma$ can always be found. By Stokes's Theorem, we have that
	\begin{align}
	\int_{x_0}^{x_1} \theta
	&=
	\int_{[x_0,0] \subset \RR} \theta + 
	\int_{[0,x_1] \subset i \RR} \theta
	- \int_R \omega \nonumber \\
	& =
	f_0(0) - f_0(x_0)
	+ f_1(x_1) - f_1(0)
	+ \int_R \omega \nonumber \\
	& =
	f_1(x_1) - f_0(x_0) \label{eqn.exact-surgery}
	\end{align}
hence by the contractability of $\Gamma$, there exists a unique function $f: \Gamma \to \RR$ defined on $\Gamma$ such that $f = f_1 \coprod f_0$ outside a neighborhood of $0$, and for which $\theta|_\Gamma = df$. An identical computation shows that the entire surgery $\Gamma \cdot S^0 \subset \CC$ admits a primitive $f$ which equals $f_1 \coprod f_0$ outside a neighborhood of $0$. (Here, the notation $f_1 \coprod f_0$ is shorthand for the following: Outside of $[-A,A]^2$, $\Gamma$ has four connected components---two of them are a subset of $\RR$, and two are a subset of $i\RR$. We mean that $f$ restricted to the subset of $\RR$ is equal to $f_0$, while $f$ is equal to $f_1$ when restricted to the subset of $i\RR$.)

For $n \geq 2$, it remains to find a $\Gamma \subset \CC$, and some function $f: \Gamma \cdot S^{n-1} \to \RR$, such that $f$ agrees with $f_0$ along $\RR^n \subset \CC^n$. The same computation as in~\eqref{eqn.exact-surgery} shows that so long as $f_1(0) - f_0(0)$ is small enough (for example, less than $A^2/2$) then such a $\Gamma$ can be found.

Now we make the constant $C$ slightly more explicit, though not by much. (We do not need a precise estimate for this paper.) We set
	\eqnn
	C = \sup A^2/2
	\eqnd
where the sup runs through all $A$ for which we can find a Darboux chart of size $[-2A,2A]^{2n} \subset \CC^n$ in which $L_0$ and $L_1$ are equal to $\RR^n$ and $i\RR^n$, respectively. We conclude by noting that the area of $\Gamma$ above $[0,A] \subset \RR \subset \RR^2$ is bounded sharply by $A^2/2$.
\end{proof}

\begin{remark}
In the proof, we assumed that $f_1(p) > f_0(p)$. If $f_1(p) < f_0(p)$, the same proof follows simply by utilizing $\overline{\Gamma}$ in place of $\Gamma$.
\end{remark}

\subsection{Stabilized surgery}\label{section.stabilized-surgery}
One can modify the local picture of surgery by taking a product with another Lagrangian $L'$. Concretely, suppose that $M = M' \times M''$ is a product. Further, suppose that $L_0, L_1 \subset M$ are Lagrangians satsifying the following:
\begin{itemize}
\item There exist  (i) an open set $W'' \subset M''$, (ii) two Lagrangians $L_0'', L_1'' \subset W''$ that intersect transversally at exactly one point $p''$, and (iii) a Lagrangian $L' \subset M'$ such that
	\enum
		\item $L_i \cap M' \times W''$ is equal to $L' \times L_i''$, and
		\item $L_0 \cap L_1 = L' \times \{p''\}$.
	\enumd 
\end{itemize}
Then there is an embedded Lagrangian submanifold as follows:

\begin{notation}\label{notation.stabilized-surgery}
We let $L_0 \sharp_{L' \times \{p''\}} L_1 \subset M' \times M''$ denote the set
	\eqn
	(L_0 \setminus M' \times W'') \bigcup (L_1 \setminus M' \times W'' ) \bigcup L' \times (L_0'' \sharp_{p''} L_1'').
	\eqnd
We call this the {\em surgery of $L_0$ and $L_1$ along $L_0 \cap L_1$.}
\end{notation}

\begin{remark}
Note that $L_0$ and $L_1$ need not be product Lagrangians themselves, but only in a neighborhood of their intersection.
\end{remark}

\begin{remark}
If $L'$ is an eventually conical Lagrangian, and if $L_0, L_1$ are both eventually conical in the product $M' \times M''$, then $L_0 \sharp_{L' \times \{p\}} L_1$ is also eventually conical in the product. (See Definition~\ref{defn.conical-in-product}.)
\end{remark}

\section{The Lagrangian cobordism associated to a surgery}
As before, fix two transverse Lagrangians $L_0, L_1$ with a unique intersection point $p \in L_0 \cap L_1$. We fix the same $A$ and $\Gamma$ as in Section~\ref{section.surgery}. The main goal of this section is to prove the following:

\begin{prop}\label{prop.surgery-cobordism}
There exists an element $\sigma \in H^0(L_1;\ZZ) \times H^2(L_1;\ZZ/2\ZZ)$ and a Lagrangian cobordism $Q: L_0 \sharp L_1 \to L_0$ such that the mapping cone of $Q$ has a vertically collared end collared by $L_1^{\sigma}$. (See Definition~\ref{defn.vertically-collared} and Notation~\ref{notation.L-sigma}.)
\end{prop}

To prove Proposition~\ref{prop.surgery-cobordism} we explicitly construct $Q$; we first learned the construction from~\cite{biran-cornea} (there, the cobordism is called a {\em trace} of the surgery). We present a slightly modified version to account for eventually conical surgeries. We construct three different subsets 
	\eqnn
	Q^{(0)}, Q^{(1)}, Q^{(2)}, Q^{(3)} \subset T^*\RR \times M.
	\eqnd. 
The last of these, $Q^{(3)} =: Q$, will be the cobordism we seek.

\subsection{$Q^{(0)}$}
\begin{notation}\label{notation.L-0}
Fix the usual diffeomorphism $T^*\RR \cong \CC$, and let 
	\eqnn
	i \RR_{>0} = \{(x,y) \, | \, x = 0, y > 0\},
	\qquad
	\RR_{>0} := \{(x,y) \, | x > 0, y = 0 \}.
	\eqnd
We define
	\eqnn
	Q^{(0)}
	:=
	\{(0,0)\} \times L_0 \sharp_p L_1
	\coprod
	\RR_{>0} \times L_0
	\coprod
	i\RR_{>0} \times L_1
	\subset
	T^*\RR \times M.
	\eqnd
\end{notation}

\begin{remark}
$Q^{(0)}$ is a submanifold with boundary, but is not closed as a subset of $T^*\RR \times M$.
\end{remark}

\subsection{$Q^{(1)}$}

\begin{notation}
Let $S^{n} \subset \RR^{n+1} \subset \CC^{n+1}$ be the unit sphere. We let $D^{n}_{+}$ denote its upper hemisphere: 
	\eqnn
	D^{n}_{+} = 
	\{
	(x_0,\ldots,x_n) \, | \, x_0 \geq 0
	\land
	\sum_{i=0}^n x_i^2 = 1.
	\}
	\eqnd
We will fix a small open neighborhood $D^n_{+} \subset U(D^n_{+}) \subset S^n$, namely the set of those $x \in S^n$ whose $0$th coordinate satisfies $x_0 > -a$ for some fixed, small real number $a>0$.
\end{notation}

\begin{notation}\label{notation.Gamma-disk}
We let $\Gamma\cdot D^n_{+} \subset \CC^{n+1}$ denote the set
	\eqnn
	\{
	(z x_0,\ldots,z x_n)
	\, | \,
	z \in \Gamma 
	\land
	(x_0,\ldots,x_n) \in D^n_{+}
	\}.
	\eqnd
Likewise, $\Gamma \cdot U(D^n_+)$ denotes the obvious analogue.
\end{notation}

\begin{remark}\label{remark.pi-0-Gamma-disk}
Consider the projection $\pi_0 : \CC^{n+1} \to \CC$ given by $(z_0,\ldots,z_n) \mapsto z_0$. Then the image $\pi_0 (\Gamma \cdot D^n_{+})$ is the convex hull generated by the origin and the curve $ \Gamma \subset \CC$.
\end{remark}

\begin{remark}
The same reasoning as in the proof of Proposition~\ref{prop.local-surgery} shows that $\Gamma\cdot D^n_{+}$ is a Lagrangian submanifold of $\CC^{n+1}$ diffeomorphic to $\RR \times D^n$---in particular, it is contractible. Moreover, $\RR \times D^n$ is a smooth manifold with boundary $\RR \times S^{n-1}$; the boundary of $\Gamma \cdot D^n_{+}$ is precisely the fiber of $\pi_0$ above the origin $0 \in \CC$. Note that (because we have chosen $\Gamma, A$ as in Section~\ref{section.surgery}) this boundary is precisely the surgery of $\RR^n$ with $i\RR^n$.
\end{remark}

\begin{remark}\label{remark.brane-structure-on-local-cobordism}
Note that the following diagram commutes:
	\eqnn
	\xymatrix{
		&	S^{n} \ar[r] & 	\RR^{1+n} \ar[r] & U_{1+n}/O_{1+n} \\
	S^{n-1} \ar[ur] \ar[r] & \RR^{n} \ar[ur] \ar[r] & U_n / O_n \ar[ur]^{\oplus \RR}
	}
	\eqnd
where the map $U_n/O_n \to U_{1+n}/O_{1+n}$ is the same stabilizing map as in Notation~\ref{notation.GrLag}. In particular, consider the Gauss map on $\Gamma \cdot S^n$ induced by the homotopy equivalence $S^n \simeq \Gamma \cdot S^n$, and restrict it to $\Gamma \cdot D^n_+$. This Gauss map agrees with the Gauss map from $\Gamma \cdot S^{n-1} = \pi_0^{-1}(0)$ utilized in the proof of Proposition~\ref{prop.brane-surgery}. In particular, the brane structure on $\Gamma \cdot D^n_+$ restricts to the brane structure on $\Gamma \cdot S^{n-1} = \RR^n \sharp i \RR^n$ described in that proof.
\end{remark}

Now let us choose a Darboux chart $\phi: V \into \CC^n$ as in Proposition~\ref{prop.local-model} about $p = L_0 \cap L_1$. Again denoting the same $A$ as always, let
	\eqnn
	j: [-A,A]^2 \cong [-A,A] \times i[-A,A] \into \CC
	\eqnd
be the obvious inclusion. 

\begin{notation}\label{notation.L-1}
We define $Q^{(1)}$ to be the union
	\eqnn
	\left(
	Q^{(0)}
	\setminus
	[-A,A]^2 \times V
	\right)
	\bigcup
	(j \times \phi^{-1})
	\left(
	\Gamma \cdot U(D^n_+) \cap ([-A, A]^2 \times \phi(V))
	\right)
	\eqnd
\end{notation}

\begin{remark}
Informally, $Q^{(1)}$ is obtained from $Q^{(0)}$ by replacing a neighborhood of $(0,p) \in Q^{(0)} \subset T^*\RR \times M$ with $\Gamma \cdot U(D^n_+)$ (see Notation~\ref{notation.Gamma-disk}).
\end{remark}

\begin{remark}
Let $\pi: T^*\RR \times M \to T^*\RR$  be the projection to the first factor. Then $\pi(Q^{(1)}) = \pi_0(\Gamma \cdot D^n_+)$. (See Remark~\ref{remark.pi-0-Gamma-disk}.)
\end{remark}

\begin{remark}
Note that $Q^{(1)}$ is an eventually conical Lagrangian in $T^*\RR \times M$. To see this, we note that $Q^{(0)}$ is eventually conical, while $Q^{(1)}$ is obtained by altering $Q^{(0)}$ is a bounded neighborhood. 
\end{remark}

\begin{remark}
Note that $Q^{(1)}$ is not closed as a subset of $T^*\RR \times M$. 
\end{remark}

\begin{remark}\label{remark.fibers-of-pi}
Let us describe the fibers of the map $\pi: Q^{(1)} \to T^*\RR$, which projects a point $(z, x) \in T^*\RR \times M$ to $z$. 
\begin{enumerate}[(a)]
\item Let $z =0$. Then $\pi^{-1}(z) = L_0 \sharp L_1$.
\item Let $z \in (0,A) \subset \RR \subset T^*\RR$. Then the fiber is equal to a set obtained from $L_0$ by removing an open ball containing $p$.
\item Let $z \in [A, \infty) \subset \RR \subset T^*\RR$. The fiber is $L_0$.
\item\label{item.vertical-collar-Q} Likewise, if $z \in i(0,A) \subset i\RR \times T^*\RR$, the fiber above $z$ is a set obtained from $L_1$ by removing an open ball containing $p$. If $z \in i[A, \infty)$, then $\pi^{-1}(z)$ is $L_1$.
\item Finally, let $z$ be a point point along $\Gamma$ which is not purely imaginary, nor real. Then $\pi^{-1}(z)$ is a single point in some small neighborhood of $p$. If $z$ is a point on the interior of the convex hull of $\Gamma$ and the origin, then $\pi^{-1}(z)$ is diffeomorphic to a sphere $S^{n-1} \subset M$ contained in a small ball about $p$.  
\end{enumerate}
We leave it to the reader to analyze the analogous fibers of points above the third quadrant of $T^*\RR \cong \CC$.
\end{remark}

\subsection{$Q^{(2)}$}

To construct $Q^{(2)}$, first consider the curve $\beta \subset T^*\RR$ given as the graph of $d({\frac 1 2 }q^2)$---that is, under the standard trivialization of $T^*\RR \cong \RR^2$, the locus $\{(q,q)\}$. 

\begin{notation}\label{notation.product-darboux-weinstein-charts}
Let us choose a Darboux-Weinstein chart of $\beta \subset T^*\RR$, and of $L_0 \sharp L_1 \subset M$. Then the direct product of these charts is a Darboux-Weinstein chart for $\beta \times L_0 \sharp L_1$:
	\eqn\label{eqn.product-darboux-weinstein}
	T^* \RR \times T^*(L_0 \sharp L_1) \supset W \xra{\phi} T^*\RR \times M
	\eqnd
As indicated in~\eqref{eqn.product-darboux-weinstein}, we denote the by chart $(W,\phi)$.
\end{notation}

\begin{remark}\label{remark.submersion}
Consider the induced map
	\eqnn
	p:
	\phi(W) \cap Q^{(1)}
	\xra{\phi^{-1}}
	T^*(\RR \times L_0 \sharp L_1)
	\to
	\RR \times L_0 \sharp L_1.
	\eqnd
(The last map is the projection to the zero section.) This composition is a submersion.
\end{remark}

\begin{remark}\label{remark.submersion-W}
Moreover, the submersion can be explicitly understood away from the set $T^*\RR \times U_p$, where $U_p \subset M$ is a well-chosen neighborhood of $p$. By Remark~\ref{remark.fibers-of-pi}, and by the fact that we have chosen a product Darboux-Weinstein chart (Notation~\ref{notation.product-darboux-weinstein-charts}), the map $p$ sends a point $(z_0, x) \in i\RR \times L_0 \sharp L_1$ to the point $(z_0',x)$ where $z_0'$ is the point on $\beta$ closest to $z_0$.
\end{remark}

\begin{remark}\label{remark.LL2-ham-isotopy}\label{remark.B-in-beta}
By shrinking $W$ if necessary, we conclude that $\phi^{-1}(\phi(W) \cap Q^{(1)})$ is the graph of an exact 1-form defined on $B \times L_0 \sharp L_1$ where $B$ is some open subset of $\beta$ containing the origin. In particular, there is a Hamiltonian isotopy from $\phi(W) \cap Q^{(1)}$ to $B \times L_0 \sharp L_1$. By Remark~\ref{remark.submersion-W}, we conclude that this isotopy can be chosen to be constant in the $L_0 \sharp L_1$ factor.  
\end{remark}

\begin{notation}
We let $Q^{(2)}$ be the Lagrangian obtained from $Q^{(1)}$ by applying (in an open neigborhood of $\{(0,0)\} \times L_0 \sharp L_1$) the isotopy from Remark~\ref{remark.LL2-ham-isotopy}. 
\end{notation}

\begin{remark}\label{remark.LL2-conical}
$Q^{(2)}$ is a Lagrangian submanifold of $T^*\RR \times M$ which, away from the closure of the first quadrant $\RR_{\geq 0} \times \RR_{\geq 0} \times M \subset \RR \times \RR \times M \cong T^*\RR \times M$, is equal to $\beta_{<0} \times L_0 \sharp L_1$. Here, $\beta_{<0}$ is the set $\{(q,q) \, | q < 0 \}$. 

Moreover, there is compact set $K \subset M$ of $p \in L_0 \cap L_1$ such that outside of $[0,A] \times i[0,A] \times K \subset \CC \times M$, $Q^{(2)}$ is eventually conical in the product $T^*\RR \times M$. (See Definition~\ref{defn.conical-in-product} and Remark~\ref{remark.fibers-of-pi}.)
\end{remark} 

\subsection{$Q^{(3)}$}
Finally, let $h: (-\infty,0] \to \RR$ be a smooth function such that $h=0$ near $-\infty$, so that $\grph(dh) \subset T^*\RR$ equals $\beta$ in a neighborhood of $B$. (See Remark~\ref{remark.B-in-beta}.) 

\begin{notation}\label{notation.LL3}
We let $Q^{(3)} \subset T^*\RR \times M$ denote the Lagrangian obtained by gluing $\grph(dh) \times L_0 \sharp L_1$ to $Q^{(2)}$ along $B \times L_0 \sharp L_1$. (See Figure~\ref{figure.LL3}.)
\end{notation}

\begin{figure}
\begin{tabular}{lllll}
		\raisebox{-.1\height}
		{\includegraphics[height=1in]{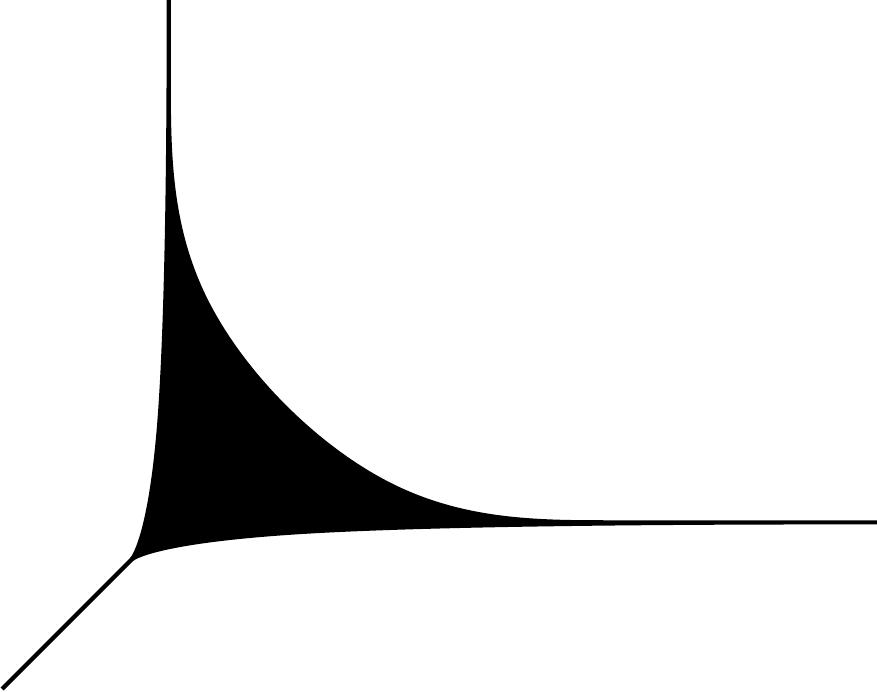}}
		& $\qquad$&
		{\includegraphics[height=0.9in]{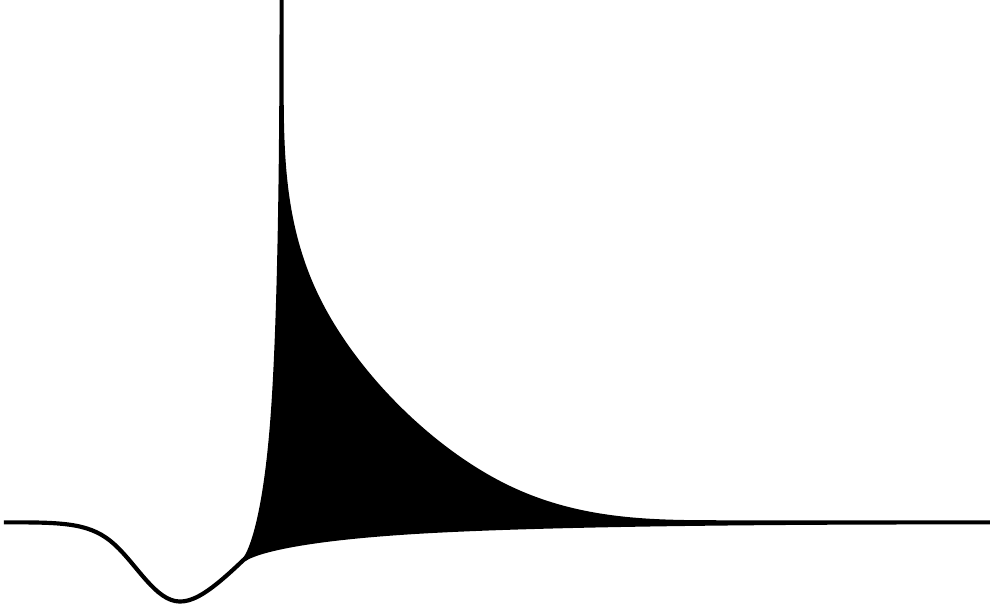}}
		& $\qquad$&
		\raisebox{-.4\height}{\includegraphics[height=1.5in]{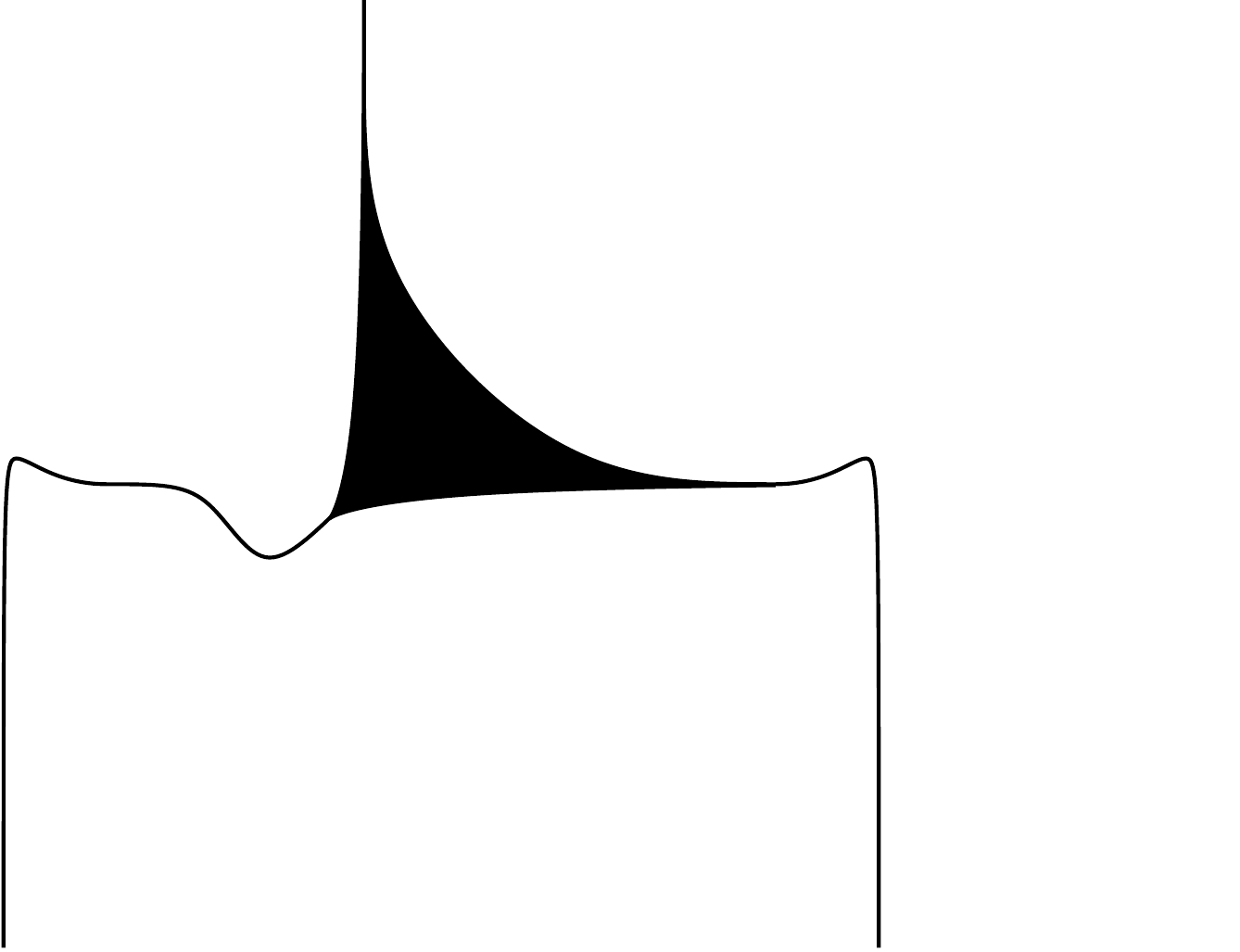}}
\end{tabular}
\caption{
On the left, an image of $Q^{(2)}$ projected onto $T^*\RR$. 
In the middle, an image of $Q^{(3)}$ projected onto $T^*\RR$. 
To the right, an image of the kernel of the resulting cobordism. Note that it is vertically collared when $p>>0$ by the brane $L_1$, and the brane structure there agrees with that of the stabilization $L_1^\dd$.
}\label{figure.LL3}
\end{figure}

\begin{remark}\label{remark.LL3-cobordism}
For $q_0 <<0$ , $Q^{(3)} \cap T^*(-\infty,q_0) \times M$ is equal to  $(-\infty, q_0) \times L_0 \sharp L_1$. For $q_1 >>0$, $Q^{(3)}$ is collared by $L_0$. That is, $Q^{(3)}$ is a Lagrangian cobordism from $L_0 \sharp L_1$ to $L_0$.
\end{remark}

\begin{remark}\label{remark.LL3-exact}
$Q^{(2)}$ is an exact Lagrangian. 
Note that $h$ can be chosen so that $Q^{(3)}$ is also an exact Lagrangian, with primitive along $q_0<<0$ given by $0 \oplus f_{L_0 \sharp L_1}$ where $f_{L_0 \sharp L_1}$ is the primitive on the surgery. (See Proposition~\ref{prop.surgery-primitive}.)
\end{remark}

\begin{remark}\label{remark.LL3-product-with-curves}
By construction, there exists a bounded neighborhood $O \subset M$ containing $p = L_0 \cap L_1$ such that, in the complement of $T^*\RR \times O$, the Lagrangian $Q^{(3)}$ is equal to the folloing:
	\eqnn
	\gamma_0 \times (L_0 \setminus O) 
	\coprod
	\gamma_1 \times (L_1 \setminus O)
	\coprod
	\gamma_2 \times (L_0 \sharp L_1 \setminus O)
	\eqnd 
where $\gamma_i$ are (non-compact) curves in $T^*\RR$.
\end{remark}

\subsection{Proof of Proposition~\ref{prop.surgery-cobordism}}

\begin{proof}[Proof of Proposition~\ref{prop.surgery-cobordism}.]
It follows from Remark~\ref{remark.LL2-conical} that $Q^{(3)}$ is eventually conical in the product. Moreover, Remark~\ref{remark.LL3-exact} shows $Q^{(3)}$ can be given a primitive collared by $f_{L_0 \sharp L_1}$ and $f_{L_0}$. Finally, it is clear that $Q^{(3)}$ admits a brane structure--$Q^{(2)}$ does because it is obtained by isotoping a brane, and $Q^{(3)}$ is obtained by gluing on $\grph(dh) \times L_0 \sharp L_1$.

Hence by Remark~\ref{remark.LL3-cobordism}, $Q^{(3)}$ is indeed a morphism in $\lag(M)$ from $L_0 \sharp L_1$ to $L_0$. The last assertion about vertical collaring follows from Remark~\ref{remark.fibers-of-pi}~\ref{item.vertical-collar-Q}.
\end{proof}

\subsection{The Lagrangian cobordism associated to a stabilized surgery}
Now let us assume we are in the situation of Section~\ref{section.stabilized-surgery}.

\begin{prop}\label{prop.stabilized-surgery-cobordism}
Stabilized surgery induces a morphism in $\lag(M)$ from $L_0 \sharp_{L' \times \{p\}} L_1$ to $L_0$. The mapping cone of this morphism has a vertical collar, collared by $L_1$.
\end{prop}

\begin{proof}
Let $Q^{(3)}$ denote the cobordism from $L_0' \sharp_p L_1'$ to $L_0'$ (see Notation~\ref{notation.LL3}). Let $\gamma_0, \gamma_1, \gamma_2 \subset T^*\RR$ be the curves from Remark~\ref{remark.LL3-product-with-curves}, and let $Z = Q^{(3)} \cap T^*\RR \times O$ (again in the notation of Remark~\ref{remark.LL3-product-with-curves}). Then the desired Lagrangian cobordism is constructed as follows:
	\eqnn
	\gamma_0 \times (L_0 \setminus M' \times W'') 
	\bigcup
	\gamma_1 \times (L_1 \setminus M' \times W'')
	\bigcup
	\gamma_2 \times (L_0 \sharp_{L' \times p''} L_1 \setminus M' \times W'')
	\bigcup
	Z \times L''. 
	\eqnd
The vertical collaring is obvious. The verification that the cobordism admits a brane structure (possibly after acting on $L_1$ by some group element $\sigma$) is straightforward and similar to our previous arguments, so we leave it to the reader.
\end{proof}

\section{Filtrations from vertically collared branes}

\begin{theorem}\label{theorem.vertical-collar-equivalence}
Let $L \subset T^*\RR \times M$ be an object of $\lag(M)$ which is vertically collared. Further assume that when writing $L$ as in~\eqref{eqn.vertically-collared}, there is a unique $i$ such that $\gamma_i$ attains an arbitrarily large, positive $p$-coordinate in $T^*\RR$. Then $L$ is equivalent to $X_i$ in $\lag(M)$.
\end{theorem}

\begin{remark}
The theorem holds true ``upside down'' as well---if there is a unique $i$ for which $\gamma_i$ attains an arbitrarily large and negative $p$-coordinate, then $L$ is equivalent to $X_i$.
\end{remark}

\begin{figure}
		\[
			\xy
			\xyimport(8,8)(0,0){\includegraphics[width=5in]{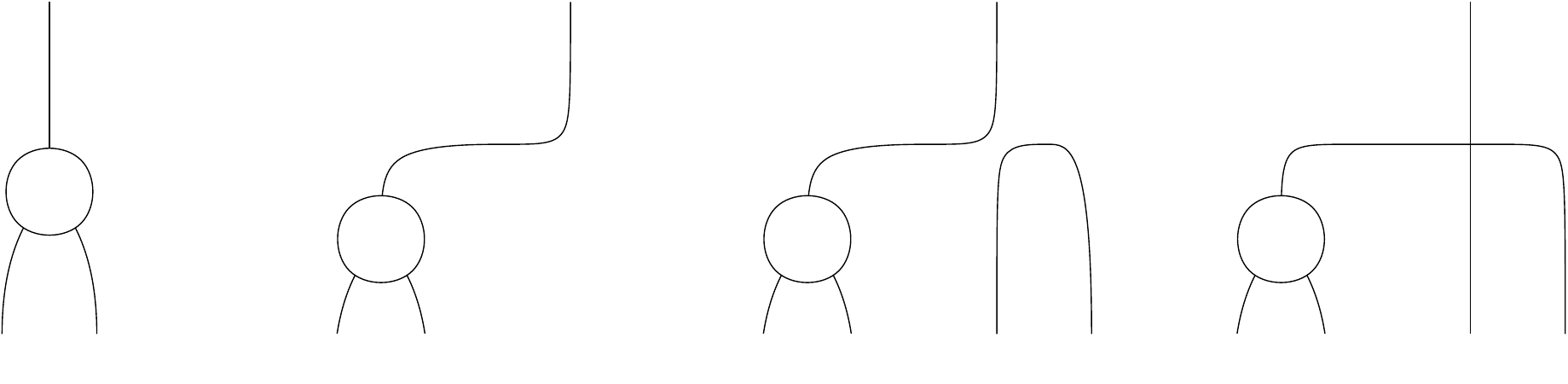}}
			,(0.2,0.4)*+{L}
			,(1.9, 0.4)*+{L'}
			,(5.4,0.4)*+{\gamma' \times X_0}
			,(6.55,0.4)*+{L_1}
			,(7.5,0.4)*+{L_0}
			\endxy
		\]
\begin{image}\label{figure.vertical-collar-argument}
Some of the Lagrangians involved in the proof of Theorem~\ref{theorem.vertical-collar-equivalence}, all projected to $T^*\RR$. 
\end{image}
\end{figure}

\begin{proof}[Proof of Theorem~\ref{theorem.vertical-collar-equivalence}.]
We refer the reader to Figure~\ref{figure.vertical-collar-argument} for a pictorial summary of this argument. 

Let $L \subset T^*\RR \times M$ and let $\gamma_0 \times X_0$ be the vertical collar. Using an eventually linear Hamiltonian isotopy, we may isotope $L$ to a Lagrangian $L'$ satsifying the following property: There exists $q_0 \in \RR$ such that every point of $L'$ has $q$-coordinate $q \leq q_0$, and
	\eqnn
	L' \cap \{(q,p) \, | \, p \geq 0\}
	\times M
	=
	\{(q_0, p) \, | \, p \geq 0 \} \times X_0
	\subset T^*\RR \times M.
	\eqnd
Since eventually linear Hamiltonian isotopies induce equivalences in $\lag(M)$, we have that $L \simeq L'$. (Proposition~\ref{prop.linear-isotopy-equivalence}.) Now let us fix some connected, non-compact curve $\gamma' \subset T^*\RR$ such that $\gamma'$ only has points with $q \geq q_0$ and $p \leq 0$, and such that $\gamma'$ is eventually conical. Then the disjoint union
	\eqnn
	L' \coprod \gamma' \times X_0
	\eqnd
is equivalent to $L'$ in $\lag(M)$. To see this, note that $\gamma'$ can be Hamiltonian-isotoped to avoid $\RR \subset T^*\RR$, hence $\gamma'\times X_0$ is a zero object (Proposition~\ref{prop.lambda-avoiding-zero-objects}). And if two branes can be separated over $\RR$ (e.g., $L'$ and $\gamma' \times X_0$ are fibered over disjoint open sets of $\RR$), their disjoint union is the coproduct in $\lag(M)$ (Proposition~\ref{prop.disjoint-union-coproduct}).

Now we note that $\gamma'$ can be chosen so that $L' \coprod (\gamma' \times X_0)$ is a stabilized surgery of two Lagrangians: (i) $L_0 = X_0^\dd$ is the stabilization of $X_0$, and (ii) $L_1$ is a Lagrangian which is equal to $L$ where $q< q_0$, but where $q \geq q_0$, equals $\gamma'' \times X_0$ where $\gamma''$ is a curve with $p < 0$. Note that by construction of $\lag(M)$, we have that $L_0 \simeq X_0$ in $\lag(M)$ (Remark~\ref{remark.stabilization-equivalence}). Further, the intersection of $L_0$ with $L_1$ is equal to $\{x\} \times X_0$ for a unique $x \in T^*\RR$. 

By Proposition~\ref{prop.stabilized-surgery-cobordism}, we know there exists a Lagrangian cobordism
	\eqn\label{eqn.Q}
	Q: L_0 \sharp_{\{x\} \times X_0} L_1 \to L_0.
	\eqnd
(One can check that, by our definition of $L_1$, one need not alter the brane structure by a group element $\sigma.$) Note that $\kernel(Q)$ is vertically collared by $L_1$---in fact, one can arrange so that for $p_0$ large enough,	
	\eqnn
	\kernel(Q) \cap \{(q,p) \, | p \geq p_0\}
	=
	\{(q_0',p) \, | p \geq p_0\} \times L_1
	\eqnd
where $q_0' \in \RR$ is some fixed real number. Then one can isotope $\kernel(Q)$ into a Lagrangian $C$ such that
	\eqnn
	C \cap \{(q,p) \, | p \geq -\epsilon\}
	=
	\{(q_0',p) \, | p \geq -\epsilon \} \times L_1
	\eqnd
for some real number $\epsilon>0$.  In particular, because $L_1$ avoids $\RR \times \Lambda \subset T^*\RR \times M$, we have that $C$ avoids $\RR \times \RR \times \Lambda \subset T^*\RR \times T^*\RR \times M$---that is, $C$ is a zero object (Proposition~\ref{prop.lambda-avoiding-zero-objects}). 

This means that $\kernel(Q) \simeq 0$ in $\lag(M)$; in other words, the map~\eqref{eqn.Q} is an equivalence (Proposition~\ref{prop.cone-zero-equivalence}). Since $L_0 \simeq X_0$ and $Q: L_0 \sharp_{\{x\} \times X_0} L_1 \simeq L$, the result follows.
\end{proof}

\subsection{Proof of Theorem~\ref{theorem.main}.}

\begin{proof}[Proof of Theorem~\ref{theorem.main}.]
Let $Q: L_0 \sharp_p L_1 \to L_0$ be the morphism construction in Proposition~\ref{prop.surgery-cobordism}. Let $\kernel(Q)\subset T^*\RR \times M$ be the mapping cone (See Section~\ref{section.mapping-cone} and Figure~\ref{figure.LL3}.). Then $\kernel(Q) $ satisfies the hypotheses of Theorem~\ref{theorem.vertical-collar-equivalence}, and we have an equivalence $\kernel(Q) \simeq L_1$.
\end{proof}

\subsection{Filtrations on vertically collared Lagrangians}\label{section.filtrations}
Let $L \subset T^*\RR \times M$ be a vertically collared Lagrangian. Let $K$ be the compact set guaranteed in Definition~\ref{defn.vertically-collared}; we now write
	\eqn\label{eqn.vertically-collared-prime}
	L \setminus (K \times M) = 
	\left(\coprod_{i \in 1, \ldots, N} \gamma_i \times X_i \right)
	\coprod
	\left(\coprod_{j \in 1, \ldots, N'} \gamma_j' \times X_j \right)
	\eqnd
where each $\gamma_i$ is a connected, non-compact curve in $T^*\RR\setminus K$ consisting of elements with $p$-coordinate positive, while $\gamma_j'$ are connected, non-compact curves in $T^*\RR \setminus K$ consisting of elements with negative $p$ coordinate. (Either or both of $N, N'$ may equal 0.) We order $\gamma_1,\ldots,\gamma_N$ according to a counterclockwise orientation of the plane, and likewise for $\gamma_1',\ldots,\gamma_{N'}'$. See Figure~\ref{figure.vertically-collared}. This geometric assumption yields the following algebraic consequence in $\lag(M)$:

\begin{theorem}\label{theorem.vertical-collar-filtration}
Under the assumption~\eqref{eqn.vertically-collared-prime}, there exists a filtration of $L$
	\eqnn
	0 \to L_{N} \to \ldots \to L_1 = L
	\eqnd
such that the associated graded pieces $L_i / L_{i+1}$ are equivalent to the branes $X_i$. Likewise, there is a filtration of $L$
	\eqnn
	0 \to L_{N'} \to \ldots \to L_1' = L
	\eqnd
such that the associated graded pieces $L_i' / L_{i+1}'$ are equivalent to the branes $X_i'$.
\end{theorem}

By the exactness of the functor $\Xi$, we conclude:

\begin{corollary}
In $\finite(\fuk(M)$, the modules represented by the $L_{i}$ above fit into the same filtration of the module represented by $L$, and the associated gradeds of the filtration are given by the modules represented by the $X_i$. Likewise for $L_{i}'$ and $X_i'$. 
\end{corollary} 

\begin{remark}
By rotating the exact sequences of the filtration in Theorem~\ref{theorem.vertical-collar-filtration}, one also obtains a filtration of $X_1$ whose associated gradeds are given by suspensions $\Sigma X_i$.
\end{remark}

\begin{figure}
$
			\xy
			\xyimport(8,8)(0,0){\includegraphics[width=4.5in]{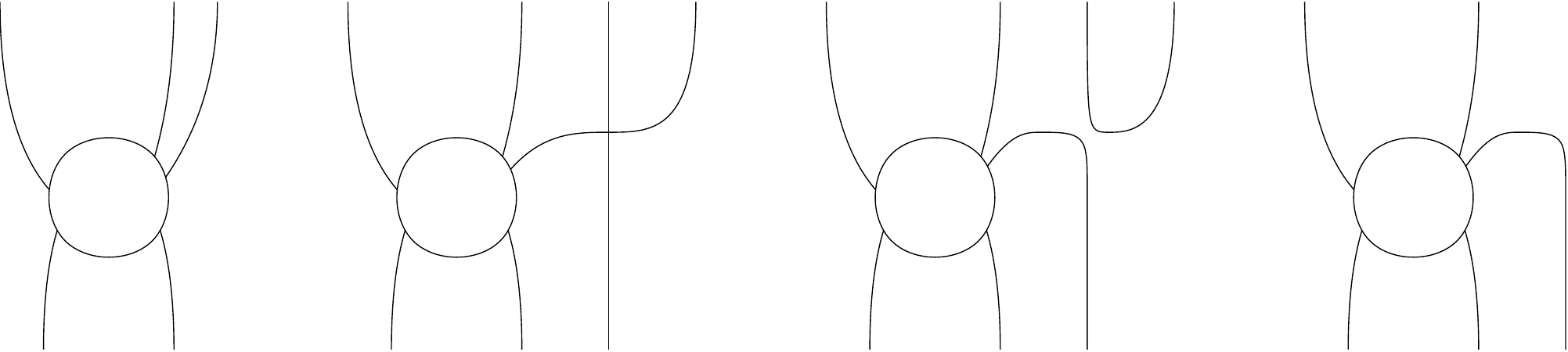}}
			,(0.5,-1)*+{L_i}
			,(1.3,8.2)*+{X_i}
			,(2.3,-1)*+{\phi(L_i)}
			,(3.4,-1)*+{\Omega X_i}
			,(5.2,-1)*+{L_{i+1} \coprod A_i}
			,(7.2,-1)*+{L_{i+1}}
			\endxy
$
\caption{
A summary of the Lagrangians involved in the proof of Theorem~\ref{theorem.vertical-collar-filtration}.
}\label{figure.filtration-argument}
\end{figure}

\begin{proof}[Proof of Theorem~\ref{theorem.vertical-collar-filtration}.]
Let $L = L_1$. Assume we have inductively defined $L_i$ with base case $i=1$. Choose a linear Hamiltonian isotopy $\phi_i$ such that $\phi_i(L_i)$ has the following property: When $q$ is large enough, $\phi_i(L_i)$ is equal to a copy of $\beta_i \times X_i$, where $\beta_i$ is some eventually conical curve tending to a large, positive $p$ coordinate.  (See Figure~\ref{figure.filtration-argument}.)

Consider a stabilization of $X_i$ with grading one more than the grading of $X_i^{\dd}$; we call this stabilization $\Omega X_i$. One can arrange for $\Omega X_i$ to intersect $\phi_i(L_i)$ uniquely at $\{x\} \times X_i$ for some $x \in T^*\RR$. The stabilized surgery yields a disjoint union of two Lagrangians; one component, which we will call $A_i$, is the product of some curve in $T^*\RR$ with $X_i$, and $A_i$ is equivalent to a zero object because it does not intersect $\RR \times \Lambda$. The other component, we will call $L_{i+1}$. By Theorem~\ref{theorem.main}, we obtain an exact sequence
	\eqnn
	\Omega X_i \to L_i \sharp_{\{x\} \times X_i} \Omega X_i \to L_{i}.
	\eqnd
Because the surgery is equivalent to $L_{i+1}$, rotating the exact sequence yields the exact sequence
	\eqnn
	L_{i+1} \to L_i \to X_i
	\eqnd
and the first filtration follows. The proof of the second filtration by $X_i'$ is similar so we omit it.
\end{proof}

\subsection{Compatibilities with the $s$-dot construction}\label{section.conjectures}
In this last section, we remark on some relations between the structures observed in this paper and in the works~\cite{lurie-waldhausen,dyckerhoff-kapranov,biran-cornea}. We assume the reader is familiar with the $s$-dot construction for stable $\infty$-categories.

First we note that the $s$-dot construction naturally defines a colored planar $\infty$-operad. Colors are given by (equivalence classes of) objects in the stable $\infty$-category $\cC$, and a $k$-ary operation from $(X_k, \ldots, X_1)$ to $X$ is given by a $k$-step filtration of $X$ 
	\eqnn
	Y_k \to Y_{k_1} \to \ldots \to Y_1 \simeq X
	\eqnd
whose associated gradeds are equipped with identifications
	\eqnn
	Y_i / Y_{i-1} \simeq X_i.
	\eqnd
Put another way, the space of $k$-simplices in the $s$-dot construction is the space of $k$-ary operations.

In parallel, suppose we consider vertically collared branes with $N=k$ and $N'=1$; the collection of such branes also forms a colored planar $\infty$-operad---indeed, the picture in Figure~\ref{figure.vertically-collared} is meant to suggest that the object $X_1'$ colors the root of a tree, while $X_k,\ldots,X_1$ color the leaves. 

\begin{conjecture}\label{conjecture.s-dot}
The filtration of Theorem~\ref{theorem.vertical-collar-filtration} defines a map of colored planar $\infty$-operads to the colored planar $\infty$-operad obtained from the $s$-dot construction of $\lag(M)$.
\end{conjecture}

Finally, the filtration of Theorem~\ref{theorem.vertical-collar-filtration} is obtained by rotating collared ends below and above the zero section $\RR \subset T^*\RR$---this discrete operation is periodic up to a shift: If we repeat it $N+1$ times, we obtain the same collared brane we began with, but with shifted grading. This is reminiscent of the famous paracyclic action on the $s$-dot construction, and we have the following vague statement:

\begin{conjecture}
Moreover, the map is compatible with the paracyclic action on the $s$-dot construction.
\end{conjecture}

Finally, if one goes through the work of verifying the analytical details needed to construct the results of~\cite{tanaka-pairing,tanaka-exact} in the monotone setting, we conjecture:

\begin{conjecture}
The categories $T^s$ of~\cite{biran-cornea} are the discrete colored planar operads associated to the $s$-dot constructions of $\lag(M)$ and $\fukaya(M)$, and the map in Conjecture~\ref{conjecture.s-dot} recovers the map in~\cite{biran-cornea}.
\end{conjecture}

\bibliographystyle{amsalpha}
\bibliography{biblio}

\end{document}